\documentclass[11pt]{article}
\usepackage[a4paper,right = 22mm, left=22mm,
 top=25mm, bottom = 25mm]{geometry}
\usepackage{graphicx} 
\usepackage{amsmath}
\usepackage{amsthm}
\usepackage{graphicx}
\usepackage{xcolor}
\usepackage{algorithm}
\usepackage{algpseudocode}
\usepackage{amsfonts}
\usepackage{caption}
\usepackage{subcaption}
\usepackage{lipsum}
\usepackage{authblk}

\newtheorem{lemma}{Lemma}
\newtheorem{proposition}{Proposition}

\numberwithin{equation}{section}
\newtheorem{corollary}{Corollary}
\usepackage{url} 
\usepackage{enumitem}
\setlist{noitemsep}
\input{abbrev.tex}
\RequirePackage{hyperref}
\title{Flexible inner-product free Krylov methods for inverse problems}
\author[1]{Malena Sabaté Landman}
\affil[1]{Mathematical Institute, Oxford, \textit{malena.sabatelandman@maths.ox.ac.uk}}

\date{}

\begin{document}

\maketitle
\begin{abstract}
Flexible Krylov methods are a common standpoint for inverse problems. In particular, they are used to address the challenges associated with explicit variational regularization when it goes beyond the two-norm, for example involving an $\ell_p$ norm for $0<p\leq 1$. Moreover, inner-product free Krylov methods have been revisited in the context of ill-posed problems to speed up computations and improve memory requirements by means of using low precision arithmetics. However, these are effectively quasi-minimal residual methods, and can be used in combination with tools form randomized numerical linear algebra to improve the quality of the results.

This work presents new flexible and inner-product free Krylov methods, including a new flexible generalized Hessenberg method for iteration-dependent preconditioning. Moreover, it introduces new randomized versions of the methods, based on the sketch-and-solve framework. Theoretical considerations are given and numerical experiments are provided for different variational regularization terms to show the performance of the new methods. \\

\noindent \textbf{keywords}: flexible Krylov methods, inner-product free methods, randomized Krylov methods, inverse problems
\end{abstract}

\section{Introduction}\label{sec1}
Large-scale inverse problems appearing in many applications, such as imaging, signal processing, and related fields, are often solved by iterative methods. 
It is often the case that the algorithms' efficiency is limited by memory constraints and the cost of linear algebra operations, and can conversely be boosted by exploiting capabilities in modern hardware such as low precision arithmetic. This work introduces a new class of flexible, inner-product free Krylov methods that offer a balanced alternative to traditional approaches. These methods are naturally well-suited for parallelization, reduce the costs of re-orthogonalization, and have demonstrated strong performance in low-precision arithmetic.

More specifically, this paper concerns large-scale linear discrete inverse problems where the system matrix $\bfA\in \mathbb{R}^{m \times n}$ is known, and the right hand side $\bfb$ is corrupted by white Gaussian noise $\bfe$. In many applications, for example computed tomography (CT) or image and signal deblurring, the resulting linear system is ill-posed. This means that the solution might not exist due to the noise in the measurements (which is very likely in overdetermined problems) or that the solution might not be unique (always true for underdetermined problems without aditional information on the solution). Moreover, and more crucially, large changes in the  solution can be caused by small changes in the measurements $\bfb$, leading to potential noise amplification. In particular, the solution of the least-squares problem 
\begin{equation}\label{eq_LS}
   \min \| \bfA \bfx - \bfb \|_2^2,
\end{equation}
is usually very different from the exact solution, or the solution of the noiseless system, which we call $\bfx_{\text{true}}$. 

The idea of regularization, in its more wider meaning, corresponds to adding some known information of the solution into the problem formulation in order to obtain an approximation that is closer to $\bfx_{\text{true}}$ than the solution of \eqref{eq_LS}. This information can be implicit, for example, most iterative methods including Krylov methods applied to \eqref{eq_LS} have a regularizing effect if they are used in combination with early stopping. However, they display something called `semiconvergence' once the noise starts to dominate the solution as it converges to the minimizer of \eqref{eq_LS}. Moreover, there are cases where explicit properties of the solution are known in advance. For example, smoothness, sparsity, similarity to previous reconstructions or a template, etc. A very powerful framework to include this information in the problem is to consider variational regularization and solve
\begin{equation}\label{eq:var}
   \min_\bfx \left\{\| \bfA \bfx - \bfb \|_2^2 + \lambda R(\bfx)\right\},
\end{equation}
where the choice of the regularization term $R(\bfx)$ is such that unwanted properties are penalized, and the regularization parameter $\lambda$ balances the weight that one gives to the prior information with respect to the data fit.

The most used variational functional is that corresponding to Tikhonov regularization, i.e. $R(\bfx)=\|\bfx\|_2^2$, which promotes smoothness in the solution. In this case, an analytic expression for the solution exists. Since we are considering large-scale problems, we cannot typically assume that full factorizations or a direct (possibly `pseudo-') inverse of the system matrix $\bfA$ are available. In this case, hybrid Krylov subspace methods are a very efficient class of methods to find approximate solutions to \eqref{eq:var}, which project the original problem into Krylov subspaces of increasing dimension, so that only a projected problem is solved at each iteration. The advantages of these methods are as follows: (1) they only require one matrix-vector product with $\bfA$ (and possibly $\bfA\t$) per iteration, (2) the approximated solutions approach $\bfx_{\text{true}}$ in only a few iterations, (3) the regularization parameter can be efficiently computed on-the-fly at each iteration for the small projected problems. For a review on these methods, see \cite{Chung2024survey}.

In the last years, there has been a constant effort to reproduce the advantages of hybrid methods for more complex variational terms, which often appear in different applications. Here are some examples. When dealing with a sparse solution, one might want to consider $R(\bfx)=\|\bfx\|_p^p$, for $0 <p\leq 1$, or its convex relaxation $R(\bfx)=\|\bfx\|_1^1$, maybe after a transformation $R(\bfx)=\|\bfL(\bfx)\|_1^1$, e.g. \cite{Fornasier2010CS}. When dealing with signals (or images) with sharp edges, one might want to consider total variation (TV) regularization, i.e. $R(\bfx)=TV(\bfx)$ \cite{rudin1992nonlinear}. In dynamical data, where the solution has both temporal and spatial dimensions, one might want to regularize these differently, for example using group sparsity, i.e. $R(\bfx)=\|\bfx\|_{2,1}$. Moreover, application-specific regularizers might be beneficial in particular contexts such as the prior-image-registered penalized-likelihood estimation (PIPLE) \cite{Stayman2013PIRPLE} or the prior image constrained compressed sensing ({PICCS}) functionals in CT procedures with repeated scans \cite{hatamikia2023source,hastings2025PICCS}.

Even if specific nonlinear optimization methods have been developed for different functionals, a very common practice to deal with such advanced regularizers is to rely on quadratic majorizers of the original functional instead. These can then be (partially) solved using Krylov subspace methods, leading to a descent step in the original functional. Traditionally, a new Krylov subspace is built for each subproblem, leading to inner-outer schemes of iterations, which are not ideal for computational efficiency. One has to mention, however, that these methods rely on short recurrences (with some caveats for varying regularization parameters), so they are very memory efficient.

More recently, different alternatives that avoid the inner-outer scheme have appeared, which are generally based on building a single subspace of increasing dimension for the solution, and solving a small projected problem at each iteration. The methods differ in the chosen solution subspaces, the most commonly used being generalized Krylov subspaces \cite{Lanza2015GKS,Huang2017MM} and flexible Krylov subspaces \cite{Gazzola2014GAT, Chung2019lp, Gazzola2021IRW, Gazzola2021Edge}. In particular, in this paper, we focus on flexible Krylov subspace methods. Note that, in general, a drawback of such methods is that they require a large storage overhead compared to the original inner-outer schemes. For this reason, they are also sometimes used in combination with restarts \cite{Buccini2023restart, onisk2025restarts}, with a small delay in convergence. We note that, alternatively, the memory requirements can be mitigated using low precision. In CT applications, even a two-fold increase in the amount of basis vectors that can be stored (in fast memory) can make a big difference in the convergence. \\

\fbox{\begin{minipage}{0.95\textwidth} \textbf{Novelty and relevance}
\noindent This paper introduces new flexible and inner-product free Krylov methods with iteration-dependent preconditioning, as well as randomized versions in the sketch-and-solve framework. In particular, this work includes the following novelties:
\begin{itemize}[topsep=0pt]
\item a new flexible generalized Hessenberg method for general rectangular matrices,
\item a new application of the flexible CMRH (FCMRH) method to approximate solutions of least-squares problems with an added variational regularization term,
\item a new flexible inner-product free Krylov method, named flexible LSLU (FLSLU), 
\item two sketch-and-solve (S\&S-) flexible inner-product free Krylov methods, namely (S\&S~-~FCMRH and S\&S-FLSLU)
\item for all the above methods, new variants including different regularization terms in the projected problem, namely hybrid (H-) and iteratively-reweighted (IRW-) methods,
\item for each of the proposed methods, theoretical results on sufficient conditions to guarantee  monotonicity of the objective function.
\end{itemize}
This paper is centered in the context of inverse problems, where a very general class of variational regularization problems is considered (e.g. $\ell_1$, total variation, group sparsity). The new presented methods are particularly relevant to address the trade off between efficiency in terms of computational cost and memory requirements, since they are a suitable flexible alternative for low precision. Moreover, they  allow for the paralellization of some operations since they avoid inner-products, which require global communication.
\end{minipage}} \\ \\

\noindent\textbf{Notation:} For a vector $\bfv$, the corresponding $i^{\text{th}}$ component is denoted as $\bfv(i)$. For a matrix $\bfB$, the component in the $i^{\text{th}}$ row and $j^{\text{th}}$ column is denoted as $\bfB(i,j)$, and the $i^{\text{th}}$ column is denoted as $\bfb_i$. We denote the $i^{\text{th}}$ canonical vector as $\bfe_i$. For a matrix $\bfB$, we denote its $i^{\text{th}}$ singular value as $\sigma_i(\bfB)$.

\section{Iteratively-reweighted norms}\label{sec2}
This work is concerned with the optimization of linear least-square problems with an added variational regularization term of the form \eqref{eq:var}. The choice of $R(\bfx)$ is very general, the only required assumption being that it can be written in the following form
\begin{equation}\label{eq:functional}
\displaystyle \min_\bfx \underbrace{\left\{\| \bfA \bfx - \bfb \|_2^2 + \lambda \|\bfW(\bfL \bfx) \bfL \bfx \|_2^2\right\}}_{f(\bfx)},
\end{equation}
where $\bfW(\bfL \bfx)$ is a square diagonal matrix of weights that depends on the solution $\bfx$, $\bfL\in \mathbb{R}^{d\times n}$ is the regularization matrix, and $\mathcal{N}(\bfA) \cap \mathcal{N}(\bfL) = \{\bf0\}.$ Note that, if the original regularization term is not differentiable everywhere (for example, when concerning an $\ell_1$ norm valued at vectors with components equal to zero), we usually incorporate a smoothing parameter in the definition of the weights. In this case, the functional in \eqref{eq:functional} is a smoothed version of the original one in \eqref{eq:var}.

A quadratic tangent majorant of a functional $f:\mathbb{R}^n\rightarrow\mathbb{R}$ at a point $\bar\bfx$, is a quadratic function $g:\mathbb{R}^n\rightarrow\mathbb{R}$ that fulfills: (1) $g(\bfx)\geq f(\bfx) \,\forall \bfx \in \mathbb{R}^n$, (2) $g(\bar\bfx) = f(\bar\bfx)$, and (3) $\nabla g(\bar\bfx) = \nabla f(\bar\bfx)$. Ignoring constant terms that do not affect the minimization, and multiplicative terms that are incorporated in the regularization parameter with a slight abuse of notation, the functional in 
\begin{equation}\label{eq:MM}
   \min_\bfx \left\{\| \bfA \bfx - \bfb \|_2^2 + \lambda \|\bfW_k \bfL \bfx \|_2^2\right\}, \quad \bfW_k =\bfW(\bfL \bar\bfx),
\end{equation}
is a quadratic tangent majorant of \eqref{eq:functional} at $\bar \bfx$. If $\bfL \in \mathbb{R}^{n \times n}$ is invertible, the minimization in \eqref{eq:MM} can be equivalently written as 
\begin{equation}\label{eq:MM_std}
   \min_\bfs \left\{\| \bfA (\bfW_k \bfL)^{-1} \bfs - \bfb \|_2^2 + \lambda \| \bfs \|_2^2\right\}, \quad \bfx = (\bfW_k \bfL)^{-1} \bfs,
\end{equation}
where $\bfW_k$ is defined as in \eqref{eq:MM}. Note that, in \eqref{eq:MM_std}, the weights $\bfW_k$ and the regularization matrix $\bfL$, appear as right preconditioners for the original least-squares problems. Moreover, since $\bfW_k$ can change at each iteration, we say that these preconditioners are iteration-dependent. 

A similar discussion can be had for the case where $\bfL \in \mathbb{R}^{n_{L} \times n}$ is not invertible. Consider a given matrix $\bfB$ and let $\bfK$ be a matrix whose columns span the null space of $\bfB$. Then, the A-weighted pseudo-inverse of $\bfB$ is defined as 
\begin{equation}\label{eq:Apseudo}
    \bfB ^{\dagger}_{\bfA} = (\bfI - \bfK(\bfA\bfK)^{\dagger}\bfA) \bfB ^{\dagger} =\bfE\bfB ^{\dagger}, \quad \text{for}  \quad \mathcal{N}(\bfB)= \mathcal{R}(\bfK).
\end{equation}
Using the definition of the A-weighted pseudo-inverse, the minimization in \eqref{eq:MM} can be equivalently written as 
\begin{equation}\label{eq:MM_w_std} 
   \min_\bfs \left\{\| \bfA (\bfW_k \bfL)^{\dagger}_{\bfA} \bfs - (\bfb-\bfA\bfx_A) \|_2^2 + \lambda \| \bfs \|_2^2\right\}, \quad \bfx = (\bfW_k \bfL)^{\dagger}_{\bfA} \bfs + \bfx_A,
\end{equation}
where $\bfx_A$ corresponds to the part of the solution in the null space of $\bfW_k \bfL$. In this case, $(\bfW_k \bfL)^{\dagger}_{\bfA}$ can be considered as an iteration-dependent right preconditioner of the original problem. Note that, even when $\bfA$ is square, working with \eqref{eq:MM_w_std} directly leads to a preconditioned system that is not square, so in order to use methods based on the Arnoldi or the Hessenberg processes, as explained in Section \ref{sec:FH}, one needs to modify problem \eqref{eq:MM_w_std} further. Following \cite{hansen2007smoothing, Sabate2019TV}, we can consider the following augmented system
\[
\bfA [(\bfW_k\bfL)^{\dagger}_{\bfA}, \bfK ]\begin{bmatrix}
\bfs \\ \bfs_{A}
\end{bmatrix} = \bfb, \quad \text{for} \quad \bfx_A=\bfK \bfs_A,
\]
where $\bfK$ is defined in \eqref{eq:Apseudo}, 
and apply left preconditioning $[\bfL^{\dagger}, \bfK]^{T}$. Eliminating $\bfS_A$ from the previous system, we obtain the following Schur complement:
\begin{equation}\label{eq_F_square} 
    (\bfL^\dagger)^T \bfP  \bfA (\bfW_k \bfL)^{\dagger}_{\bfA} \bfs =(\bfL^\dagger)^T \bfP \bfb, \quad \text{for}\quad \bfP =\bfI-\bfA\bfK(\bfK^T\bfA \bfK)^{-1} \bfK^T \in \mathbb{R}^{n\times b},
\end{equation}
where $\bfP$ is the oblique projector onto the orthogonal complement of $\mathcal{R}(\bfK)=\mathcal{N}(\bfL)$ along $\mathcal{R}(\bfA\bfK)$.

Summarizing, for any variational problem of the form~\eqref{eq:var} which can be written in the form~\eqref{eq:functional} –where we might have added some additional smoothing–, and given a solution $\bar\bfx$,  we can write a minimization problem involving a quadratic tangent majorant of \eqref{eq:functional} at $\bar\bfx$ of the form \eqref{eq:MM}. Starting from an initial solution $\bfx_0$, this can be used to build a sequence of minimization problems corresponding to the quadratic tangent majorants of  \eqref{eq:functional} at each $\bfx_{k-1}$, which can then be solved using an iterative solver to obtain $\bfx_{k}$.  

Alternatively, one can consider the characterizations \eqref{eq:MM_std} or  \eqref{eq:MM_w_std}, and use methods which allow for flexible preconditioning, e.g. flexible Krylov methods. Note that this type of preconditioning is used to incorporate prior information in the solution space, rather than accelerating convergence, see e.g.  \cite{calvetti2005priorconditioners}.

\section{Flexible inner-product free Krylov subspace methods}\label{sec3}

In this section we revisit the flexible Hessenberg method and introduce a new generalized Hessenberg method, both of which allow for right iteration-dependent preconditioning. Moreover, we introduce a new framework of flexible solvers that approximate solutions of problems with variational regularization, using the aforementioned flexible Hessenberg factorizations, with the iteration-dependent preconditioners defined in \eqref{eq:MM_std} or \eqref{eq:MM_w_std}, and more generally motivated in Section \ref{sec2}.
 
\subsection{Flexible Hessenberg method}\label{sec:FH}
The construction of the Hessenberg basis has been known for a long time \cite{Hessenberg1940}, for more details on the implementation see \cite{sadok1999new}, including flexible variants when using an inner solver as a preconditioner \cite{Zhang2013CMRHflexible,Gu2020FCMRH}. 

In this subsection, we recall the flexible Hessenberg methods, in e.g., \cite{Zhang2013CMRHflexible}. To do this, assume that we have the following iteration-dependent preconditioners $\{\bfP_{k}\}_{k=1,\dots,k_{\max}+1}$, all in $\mathbb{R}^{n \times n}$. In practice, however, these depend on intermediate solutions of an iterative solver and are usually computed as the iterations proceed. If we apply the flexible Hessenberg method (with pivoting) to $\bfA \in \mathbb{R}^{n \times n}$ and $\bfb$ as described in Algorithm \ref{alg:Hessenberg}, given the sequence of preconditioners and $\bfx_0$, we obtain a factorization of the form:
\begin{equation}\label{eq:Hessenberg_factorization}
    \bfA \, \bfZ^{H}_{k} = \bfV^{H}_{k+1}  \, \bfH^{H}_{k+1,k}, \quad \bfZ^{H}_{k} = [\bfP_1 \bfv^{H}_1,\dots \bfP_{k+1}\bfv^{H}_{k+1}] 
\end{equation}
where the superscripts in Algorithm \ref{alg:Hessenberg} are omitted to simplify the notation, as well as the under-script in the iteration number $k_{\max}$ in \eqref{eq:Hessenberg_factorization}. Here, $\bfZ^{H}_{k}\in\mathbb{R}^{n\times k}$, $\bfV^{H}_{k+1}=[\bfv^{H}_1,\dots \bfv^{H}_{k+1}]\in\mathbb{R}^{n\times k+1}$ and $\bfH^{H}_{k+1,k}$ is upper Hessenberg.
Note that partial pivoting on the rows of $\bfV^{H}_{k+1}$ is recommended to improve  stability. If $\bfPi_k$ is the pivoting matrix, so that each of its columns is the canonical vector corresponding to the pivot, i.e. $\bfpi_i=\bfe_{\bfp(i)}$, for $\bfp$ defined in Algorithm \ref{alg:Hessenberg} (lines 1, 4 and 13), then $\bfPi_k \bfV^{H}_{k+1}$ is unit lower triangular. Note that, however, $\bfZ^{H}_{k}$ does not have any particular structure.

\begin{algorithm}[ht]
\caption{Flexible Hessenberg process with pivoting} \label{alg:Hessenberg}
\hspace*{\algorithmicindent} \textbf{Input}: $\bfA$, $\bfb$, $\bfx_0$, $k_{\max}$, $\{\bfP_{k}\}_{k=1,\dots,k_{\max}+1}$\\
\hspace*{\algorithmicindent} \textbf{Output}: $\bfZ_{k_{\max}}$, $\bfV_{k_{\max}+1}$, $\bfH_{k_{\max}+1,k_{\max}}$
\begin{algorithmic}[1]
\State Initialize $\bfp = [1, \dots, n]^{T}$
\State $\bfr_0=\bfb - \bfA\bfx_0$ 
\State Find $i$ such that $|\bfr_0(i)| = \|\bfr_0\|_{\infty}$
\State $\beta = \bfr_0(i)$; $\bfp(1) \Leftrightarrow \bfp(i)$ $^*$
\State $ \bfv_1 = \bfr_0/ \beta$; $ \bfz_1 = \bfP_1 \bfv_1$
\For{$k = 1,\dots, k_{\max}$}
\State $\bfv = \bfA \bfz_k$
\For{$j = 1,\dots,k$} 
\State $\bfH_{k+1,k}(j,k) = \bfv(\bfp(j))$ ; $\bfv = \bfv - \bfH_{k+1,k}(j,k) \bfv_j$
\EndFor
\If{$k < n$ and $\bfu \neq \bf0$}
\State Find $i \in \{k+1, \dots, n\}$ s.t. $|\bfv(\bfp(i))| = \| \bfv(\bfp(k+1:n)) \|_{\infty}$
\State $\bfH_{k+1,k}(k+1,k) = \bfu(\bfp(i))$; $\bfp(k+1) \Leftrightarrow \bfp(i)$ $^*$
\State $\bfv_{k+1} = \bfv/ \bfH_{k+1,k}(k+1,k) $; $ \bfz_{k+1} = \bfP_{k+1} \bfv_{k+1}$
\Else
\State $\bfH_{k+1,k}(k+1,k) = 0$; Stop.
\EndIf
\EndFor
\end{algorithmic} 
\vspace{-3pt}
\begin{footnotesize}
* Here, $\bfa_i \Leftrightarrow \bfa_i$ denotes swapping the elements in positions $i$ and $j$ in the vector $\bfa$.
\end{footnotesize}
\end{algorithm} 

\subsection{Generalized Flexible Hessenberg method}
A more general method which allows for non-square matrices $\bfA\in\mathbb{R}^{m\times n}$ is the generalized Hessenberg method \cite{brown2025hlslu}. Similarly to Section \ref{sec:FH}, we introduce a new flexible variant. As before, assume that $\{\bfP_{k}\}_{k=1,\dots,k_{\max}+1}$, all in  $\mathbb{R}^{n \times n_L}$, are iteration-dependent preconditioners and, given $\bfb$ and $\bfx_0$, apply the generalized flexible Hessenberg method (with pivoting) as described in Algorithm \ref{alg:gen_hessenberg}. Then, we obtain a factorization of the form:
\begin{eqnarray}\label{eq:Gen_Hessenberg_factorization}
    \bfA \, \bfZ^{GH}_{k} &=& \bfU^{GH}_{k+1}  \, \bfH^{GH}_{k+1,k}, , \quad \bfZ^{GH}_{k} = [\bfP_1 \bfv^{GH}_1,\dots \bfP_{k+1}\bfv^{GH}_{k+1}]  \nonumber \\
    \bfA^{T} \, \bfU^{GH}_{k+1} &=& \bfV^{GH}_{k+1}  \, \bfT^{GH}_{k+1},
\end{eqnarray}
where the superscripts in Algorithm \ref{alg:gen_hessenberg} are omitted to simplify the notation, $\bfH^{GH}_{k+1,k} \in \mathbb{R}^{k+1 \times k}$ is upper Hessenberg and $\bfT^{GH}_{k+1}\in \mathbb{R}^{k+1 \times k+1}$ is upper triangular. Note that partial pivoting on the rows of both $\bfV^{GH}_{k+1}$ and $\bfU^{GH}_{k+1}$ is performed to ensure  stability.  If $\bfPi^{\bfV}_k$ and  $\bfPi^{\bfU}_k$ are the respective pivoting matrix, so that each of their columns is the canonical vector corresponding to the pivot, i.e. $\bfpi^{\bfV}_i=\bfe_{\bfg(i)}$ and $\bfpi^{\bfU}_i=\bfe_{\bfq(i)}$, then $\bfPi^{\bfV}_k \bfV^{GH}_{k+1}$ and $\bfPi^{\bfU}_k \bfU^{GH}_{k+1}$ are unit lower triangular. In this case, the vectors of pivots $\bfg$ and $\bfq$ are defined in Algorithm \ref{alg:gen_hessenberg} lines 1,5 and 25 and 1,3 and 14, respectively. Note that $\bfZ^{GH}_{k}$ does not have any particular structure.

\begin{algorithm}[ht]
\caption{Flexible generalized Hessenberg process with pivoting} \label{alg:gen_hessenberg}
\hspace*{\algorithmicindent} \textbf{Input}: $\bfA$, $\bfb$, $\bfx_0$, $k_{\max}$, $\{\bfP_{k}\}_{i=k,\dots,k_{\max}+1}$\\
\hspace*{\algorithmicindent} \textbf{Output}: $\bfZ_{k_{\max}}$, $\bfV_{k_{\max}+1}$, $\bfU_{k_{\max}+1}$, $\bfH_{k_{\max}+1,k_{\max}}$, $\bfT_{k_{\max}+1}$
\begin{algorithmic}[1]
\State Initialize $\bfq = [1,\dots,m]^T$, $\bfg = [1,\dots,n]^T $
\State $\bfr_0 = \bfb - \bfA \bfx_0$, find $i$ such that $|\bfr_0(i)| = \|\bfr_0\|_{\infty}$
\State $\beta = \bfr_0(i)$; $\bfu_1 = \bfr_0/ \beta$; $\bfq(1) \Leftrightarrow \bfq(i)$ $^{*}$
\State  $\bfv= \bfA^T \bfu_1$, find $i$ such that $|\bfv(i)| = \|\bfv\|_{\infty}$
\State $\bfT_{k+1}(1,1) = \bfv(i)$; $\bfg(1) \Leftrightarrow \bfg(i)$ $^{*}$
\State $\bfv_{1} = \bfv/\bfT_{k+1}(1,1)$; $\bfz_{1} = \bfP_1 \bfv_1$
\For{$k = 1,\ldots,k_{\max}$} 
    \State $\bfu = \bfA \bfz_k$
    \For{$j = 1, \ldots,k$}
        \State $\bfH_{k+1,k}(j,k) = \bfu(\bfq(j))$; $\bfu = \bfu-\bfH_{k+1,k}(j,k)\bfu_j$
    \EndFor
    \If{$k<m$ and $\bfu \neq \bf0$}
        \State Find $i \in \{ k+1,\dots,m\}$ s.t. $|\bfu(\bfq(i))| \hspace{-2pt}= \hspace{-2pt}\|\bfu(\bfq(k+1\hspace{-2pt}:\hspace{-2pt}m))\|_{\infty}$
        \State $\bfH_{k+1,k}(k\hspace{-2pt}+\hspace{-2pt}1,k)=\bfu(\bfq(i))$; $\bfq(k+1) \Leftrightarrow \bfq(i)$ $^{*}$
        \State  $\bfu_{k+1}=\bfu/\bfH_{k+1,k}(k\hspace{-2pt}+\hspace{-2pt}1,k)$;
        \Else
        \State break
    \EndIf   
    \State $\bfv = \bfA^T \bfu_{k+1}$
    \For{$j = 1,\ldots, k$}
        \State $\bfT_{k+1}(j,k+1) = \bfv(\bfg(j))$; $\bfv = \bfv - \bfT(j,k) \bfv_j$
    \EndFor
    \If{$k<n$ and $\bfv \neq \bf0$}
        \State Find $i \in \{k+1,\dots,n\}$ s.t. $|\bfv(\bfg(i))| = \|\bfv(\bfg(k+1:n))\|_{\infty}$
        \State $\bfT_{k+1}(k+1,k+1) = \bfv(\bfg(i))$; $\bfg(k+1) \Leftrightarrow \bfg(i)$ $^{*}$
        \State  $\bfv_{k+1} = \bfv/\bfT_{k+1}(k+1,k+1)$; $\bfz_{k+1}=\bfP_{k+1}\bfv_{k+1}$
    \Else
        \State break
    \EndIf 
\EndFor 
\end{algorithmic} 
\vspace{-3pt}
\begin{footnotesize}
* Here, $\bfa_i \Leftrightarrow \bfa_i$ denotes swapping the elements in positions $i$ and $j$ in the vector $\bfa$.
\end{footnotesize}
\end{algorithm}

\subsection{FCMRH and FLSLU}
The factorizations associated to the flexible Hessenberg method \eqref{eq:Hessenberg_factorization} and its generalized variant \eqref{eq:Gen_Hessenberg_factorization} can be written in a unified form by dropping the superscripts as
\begin{equation}\label{eq:both_factorization}
    \bfA \, \bfZ_{k} = \bfU_{k+1}  \, \bfH_{k+1,k},
\end{equation}
where, for the Hessenberg method, $\bfU_{k+1}=\bfV_{k+1}$.   

Assume that the iteration-dependent preconditioning $\bfP_k$ in Algorithm \ref{alg:Hessenberg} or \ref{alg:gen_hessenberg} is constructed at each iteration as a function of the solution at the previous iteration, as defined in \eqref{eq:MM}, i.e. $\bfW_k =\bfW(\bfL \bfx_{k-1})$, where the definition of $\bfW(\cdot)$ depends on the variational functional in \eqref{eq:var} that needs to be approximated (some examples are given in Section \ref{sec55}). Then, the factorization~\eqref{eq:both_factorization} can be used to find regularized  approximate solutions to the original least-squares problem~\eqref{eq_LS}. First, note that the columns of $\bfZ_{k}$, spanning the solution space, already incorporate prior information on the properties of the solution through the variable preconditioning. Therefore, assuming $\bfx_0=\bf0$ without loss of generality, one would like to solve 
\begin{equation}\label{eq:H_projection}
   \min_{\bfy\in\mathbb{R}^k} \| \bfA \bfZ_k\bfy -\bfb \|_2^2 = \min_{\bfy\in\mathbb{R}^k} \| \bfU_{k+1} (\bfH_{k+1,k}\bfy -\beta \bfe_1) \|_2^2.
\end{equation}
If $\bfU_{k+1}$ had orthonormal columns, the least-squares problem in \eqref{eq:H_projection} would be trivially reduced to a $k+1 \times k$ problem. However, since this is not the case, we consider quasi-minimal residual methods by only approximating this minimization with 
\begin{equation}\label{eq:H_projection_approx}
   \min_{\bfy\in\mathbb{R}^k} \| \bfU^{\dagger}_{k+1}(\bfA \bfZ_k\bfy -\bfb) \|_2^2 = \min_{\bfy\in\mathbb{R}^k} \| \bfH_{k+1,k}\bfy -\beta \bfe_1 \|_2^2,
\end{equation}
so, at each iteration, we solve a small system of dimension $k+1 \times k$. In particular, if the factorization in \eqref{eq:both_factorization} is coming from the flexible Hessenberg method in Algorithm
\ref{alg:Hessenberg}, we call this method flexible CMRH (FCMRH). If, instead, it is coming from the flexible generalized Hessenberg method in Algorithm
\ref{alg:gen_hessenberg}, we call this method flexible LSLU (FLSLU).

Note that, in the case where the subspace basis vectors are constructed using the non-flexible standard or generalized Hessenberg methods, the same approximation principle behind equation \eqref{eq:H_projection_approx} is at the core of the CMRH \cite{sadok1999new,sadok2012new} (or, for inverse problems \cite{brown2024hcmrh}), and LSLU methods \cite{brown2025hlslu}. See those references for more details on the name choices.

Moreover, one can also incorporate Tikhonov regularization in the projected problem \eqref{eq:H_projection_approx} to avoid semi-convergence, i.e., solve
\begin{equation}\label{eq:hybrid_H_projection_approx}
   \min_{\bfy\in\mathbb{R}^k} \| \bfH_{k+1,k}\bfy -\beta \bfe_1 \|_2^2+\lambda\|\bfy\|_2^2.
\end{equation}
In this case, we call the methods hybrid FCMRH (H-FCMRH), and hybrid FLSLU (H-FLSLU).

Note that, even if specific properties (such as sparsity) are already promoted through the choice of basis vectors, one could further aim to approximately solve the variational problem in \eqref{eq:functional} by minimizing an (approximated) projection of the reweighted minimization in \eqref{eq:MM} for $\bfx \in \mathcal{R}(\bfZ_k)$,
\begin{equation}\label{eq:IRW_H_projection_approx}
   \min_{\bfy\in\mathbb{R}^k} \| \bfH_{k+1,k}\bfy -\beta \bfe_1 \|_2^2+\lambda\|\bfU^{LU}_k\bfy\|_2^2, \quad \text{for} \quad \bfL^{LU}_k\bfU^{LU}_k = \bfW_k \bfL\,\bfZ_k,
\end{equation}
where $\bfU^{LU}_k$ is a lower triangular matrix coming from the LU decomposition of $\bfW_k \bfL\,\bfZ_k$.
In this case, we call the methods iteratively-reweighted FCMRH (IRW-FCMRH), and iteratively-reweighted FLSLU (IRW-FLSLU).

\subsection{Randomized FCMRH and FLSLU}
The residual approximation in \eqref{eq:H_projection_approx},  \eqref{eq:hybrid_H_projection_approx} or \eqref{eq:IRW_H_projection_approx}, can sometimes give too poor of an approximation of the solution. In that case, one can borrow tools from randomized numerical linear algebra to compute alternative cheap approximations of the functionals once the solution is restricted to live in a small subspace. 

In particular, we say that a matrix $\bfS \in \mathbb{R}^{s\times m}$ is a sketching matrix for a given tall and skinny matrix $\bfB$ if we assume that it fulfills the following subspace embedding property: for $\bfr \in \mathcal{R}(\bfB)$,
\[(1-\eps)\|\bfr\|_2\leq\|\bfS\bfr\|_2\leq (1+\eps)\|\bfr\|_2,\]
for a moderate distortion factor $0<\eps\leq1$. Given $s \ll m$, this is a linear dimensionality reduction technique. For more information, see ~\cite[\S~8.1]{Martinsson_Tropp_2020}. Note that, to keep the methods inner-product free, or at least communication-aware, many sketching techniques with theoretical guarantees cannot be used. However, one can use sub-sampling, maybe in combination with local leverage scores, see e.g. \cite{Kolda2022sampling}, for each of the (row) partitions of the matrix whose columns span $\mathcal{R}(\bfB)$. In this case, one would only require inner-products for vectors of reduced (sketched) dimension. It remains future work to study the theoretical properties of using this technique. Moreover, even if some inner products were used to form the sketched residual, this might still alleviate some of the problems associated to using inner-products in iterative solvers at low precision (such as early break up of the algorithm due to under or overflow, or information loss in the reorthogonalizaton process, see \cite{brown2024hcmrh}).

Using sketching in combination with least-squares problems has been widely used, and in particular for inverse problems in \cite{Sabate2025randCMRH,chung2025randomized,Sabate2025rand}. In this case, we propose  approximately solving \eqref{eq:H_projection} using a sketching matrix $\bfS_1\in\mathbb{R}^{s_1\times m}$, for $s_1=O(k_{\text{max}})$ such that we minimize
\begin{equation}\label{eq:H_sketched_projection}
   \min_{\bfy\in\mathbb{R}^k} \| \bfS_1 (\bfA \bfZ_k\bfy -\bfb)\|_2^2, 
\end{equation}
i.e.
\[
\bfy_k = ((\bfA \bfZ_k)^T{\bfS}_1^T \bfS_1\ (\bfA \bfZ_k))^{-1} (\bfA \bfZ_k)^T\bfS_1^T \bfS_1 \bfb.
\]
We call these methods sketched-and-solve flexible CMRH (S\&S-FCMRH) and LSLU (S\&S-FLSLU).

Moreover, we can also add hybrid regularization by solving
\begin{equation}\label{eq:H_sketched_projection_hybrid}
   \min_{\bfy\in\mathbb{R}^k} \| \bfS_1 (\bfA \bfZ_k\bfy -\bfb)\|_2^2+\lambda\|\bfy\|_2^2, 
\end{equation}
Note that this is a `project-then-regularize' method. Following the notation convention in \cite{Sabate2025rand}, we call these methods sketched-and-solve hybrid flexible CMRH (S\&S-H-FCMRH) and LSLU (S\&S-H-FLSLU) 

Last, we can also consider sketch projections of the variational problem in \eqref{eq:functional} for $\bfx \in \mathcal{R}(\bfZ_k)$: 
\begin{equation}\label{eq:H_sketched_projection_IRW}
   \min_{\bfy\in\mathbb{R}^k} \| \bfS_1 (\bfA \bfZ_k\bfy -\bfb)\|_2^2+\lambda\|\bfS_2\bfW_k \bfL \bfZ_k\bfy\|_2^2. 
\end{equation}
Here we use the same notation for $\bfS_2 \in\mathbb{R}^{s_2\times n_{L}}$ as we did in the previous method to ease the reading, but note that these might have different dimensions if $\bfL$ is not square. Consistently with \cite{Sabate2025rand}, we call the latter methods sketched-and-solve iteratively-reweighted flexible CMRH (S\&S-IRW-FCMRH) and LSLU (S\&S-IRW-FLSLU).

Note that the Arnoldi process (resp. GKB) or the Hessenberg method (resp. generalized Hessenberg), for the same matrix $\bfA$ and right hand side $\bfb$, construct different sets of basis vectors for the same Krylov subspaces. However, this is not true for the flexible case, since the variable preconditioners are applied to different vectors, leading to different subspaces.

As a last note, another interesting approach to use sketching in combination with iterative methods is to solve the original restricted problem \eqref{eq:H_projection}, possibly in combination with an explicit regularizer such as $\|\bfZ_k\bfy\|_2^2$ or $\|\bfW_k \bfL \bfZ_k\bfy\|_2^2$, with a nested iterative method. In this setting, since we are solving a sequence of tall and skinny problems, sketching can be used to compute an efficient but effective preconditioning, see e.g. \cite{Sabate2025rand}. In this work we do not explore this avenue, but we note that even if  inner products might be required in the inner solver, these would be for vectors of a much smaller dimension.

\section{Theoretical considerations}\label{sec4}
The following results are based on the study of the monotonicity of the minimization functional $f(\bfx)$ in \eqref{eq:functional} evaluated at the solutions produced by the proposed methods, and of the conditions that can guarantee such monotonicity. Here, we consider the regularization parameter to be fixed ahead of the iterations.

The idea of the proof always follows the same structure. First note that, at each iteration,  we consider a the functional \eqref{eq:functional} or a quadratic tangent majorant of it at the previous solution $\bfx_{k-1}$. Then, we perform a step of an (inexact) projection method on a subspace of increasing dimension, where we know that the sequence of solution subspaces is nested. Here, we mean inexact in the sense that both inner-product free and sketch-and-solve Krylov methods are quasi-minimal residual methods, so we consider the `inexactness' to be the distance between the minimized objective function and the true residual norm. The key idea to guarantee monotonicity of the original function~\eqref{eq:functional} at each iteration $k$ is to impose that the decrease in the inexact objective function is large enough compared to a given bound on the difference between the inexact and the exact objective functions, dictated by the source of inexactness.

\begin{lemma}\label{lemma1} Given two functions $h(\bfx)$ and $\hat h(\bfx)$ such that
\[
k_2 h(\bfx) \leq \hat h(\bfx) \leq k_1 h(\bfx), \quad k_1 > 0, \  k_2 > 0,
\]
and define $ \bfx_k=\argmin_{\bfx\in\mathcal{R}(\bfZ_k)} \bar h (\bfx)$. Then, for any $\bfx_{k-1}\in \mathcal{R}(\bfZ_{k-1}) \subset \mathcal{R}(\bfZ_k)$,
\[
\frac{\hat h(\bfx_{k-1})-\hat h(\bfx_{k})}{\hat h(\bfx_{k})} \geq\frac{k_1-k_2}{k_2} \quad  \Rightarrow \quad  h(\bfx_{k-1}) \geq h(\bfx_{k}) 
\]
\end{lemma}
\begin{proof}
We want to find conditions for the  following inequality to hold
\begin{eqnarray}\label{eq_proof1}
    h(\bfx_{k-1}) - h(\bfx_{k}) &\geq& 
k_1^{-1} \hat h(\bfx_{k-1}) - k_2^{-1}  \, \hat h(\bfx_{k})  = \frac{k_2  \hat h(\bfx_{k-1})- k_1  \hat h(\bfx_{k}) }{k_1 k_2}\geq 0.
\end{eqnarray}
Since $k_1 k_2 \geq 0$, this is true if 
\begin{equation}\label{eq_proof2}
k_2  \hat h(\bfx_{k-1})- k_1  \hat h(\bfx_{k}) = k_2  (\Delta_h +\hat h(\bfx_{k}))- k_1  \hat h(\bfx_{k}) \geq 0
\end{equation}
where  $\Delta_h = \hat h(\bfx_{k-1})-\hat h(\bfx_{k})\geq0$ because
\begin{eqnarray*}
    \hat h(\bfx_{k}) = \min_{\bfx \in \mathcal{R}(\bfZ_{k})}\hat h(\bfx) \leq \hat h(\bfx_{k-1}) \quad \text{for} \quad \bfx_{k-1}\in\mathcal{R}(\bfZ_{k-1}) \subseteq \mathcal{R}(\bfZ_{k}). 
\end{eqnarray*}
Rearranging the terms in \eqref{eq_proof2}, we have that \eqref{eq_proof1} holds if:
\[\frac{\Delta_h}{\hat h(\bfx_{k})} \geq \frac{k_1-k_2}{k_2}.\]
\end{proof}

\begin{proposition}\label{prop1} Given a sequence of solutions $\{\bfx_k\}_{k=1,..	}$ obtained with FCMRH or FLSLU, then
\[\frac{\hat q_k(\bfx_{k-1})-\hat q_k(\bfx_{k})}{\hat q_k(\bfx_{k})} \geq \frac{\sigma_1^2(\bfU_{k+1}^{\dagger})}{\sigma_{k+1}^2(\bfU_{k+1}^{\dagger})}-1 = \kappa(\bfU_{k+1}^{\dagger})^2-1 \Rightarrow \|\bfA \bfx_{k-1} - \bfb\|_2^2 \geq \|\bfA \bfx_k - \bfb\|_2^2  \] 
where $\hat q_k(\bfx_{k-1})$ is the functional minimized at each iteration in the corresponding Krylov subspace, and where $\bfU_{k+1}$ is defined in  \eqref{eq:both_factorization}. Here, recall that $\sigma_i(\bfB)$ is the $i^{\text{th}}$ singular value of $\bfB$.
\end{proposition}

\begin{proof} At each iteration of FCMRH or FLSLU, the following functional in  \eqref{eq:H_projection_approx} is minimized
\begin{equation*}
   \hat q_k(\bfx) = \| \bfU_{k+1}^{\dagger} (\bfA \bfx - \bfb) \|_2^2,
\end{equation*}
where $\bfU_{k+1}$ is defined in  \eqref{eq:both_factorization}. Then, we have that:
\begin{eqnarray*}
\underbrace{\sigma_{k+1}^2(\bfU_{k+1}^{\dagger})}_{k_2 \geq 0 } \,  \underbrace{\| \bfA \bfx - \bfb \|_2^2}_{f(\bfx)} \leq   \hat q_k(\bfx) \leq  \underbrace{\sigma_1^2(\bfU_{k+1}^{\dagger})}_{k_1 \geq 0 } \, \underbrace{ \| \bfA \bfx - \bfb \|_2^2}_{f(\bfx)}.
\end{eqnarray*}
Using Lemma \ref{lemma1} with $h(\bfx)=\|\bfA\bfx-\bfb\|_2^2$ and $\hat h(\bfx) = \hat q_k(\bfx)$, the following are sufficient conditions for the solution at iteration $k$ to be a descent step in the residual norm:
\[\frac{\hat q_k(\bfx_{k-1})-\hat q_k(\bfx_{k})}{\hat q_k(\bfx_{k})} \geq \frac{k_1-k_2}{k_2}=\kappa(\bfU_{k+1}^{\dagger})^2-1.\]
\end{proof}

\begin{corollary} Given a sequence of solutions $\{\bfx_k\}_{k=1,..	}$ obtained with S\&S-FCMRH or S\&S-FLSLU, then,
\[
\frac{\hat q_k(\bfx_{k-1})-\hat q_k(\bfx_{k})}{\hat q_k(\bfx_{k})} \geq \kappa(\bfS_1)^2-1 \Rightarrow \|\bfA \bfx_{k-1} - \bfb\|_2^2 \geq \|\bfA \bfx_k - \bfb\|_2^2 
\]
where $\hat q_k(\bfx_{k-1})$ is the (sketched) functional \eqref{eq:H_sketched_projection}, minimized at each iteration in the corresponding Krylov subspace. Moreover, since at iteration $k$, $\bfx_{k} \in \mathcal{R}(\bfZ_k)$ and $\bfx_{k-1} \in \mathcal{R}(\bfZ_{k-1})\subset \mathcal{R}(\bfZ_k)$, a tighter bound can be similarly derived by considering the singular values of $\bfS_1$ restricted to the range of $[\bfA \bfZ_k,\bfb]$. 
\end{corollary}
\begin{proof}
The proof can be readily adapted from the proof of Preposition \ref{prop1}. Similar bounds can also be found in \cite{Sabate2025rand}.
\end{proof}

\begin{proposition}\label{prop2} Given a sequence of solutions $\{\bfx_k\}_{k=1,..	}$ obtained with H-FCMRH or H-FLSLU, and a fixed regularization parameter $\lambda$, we define
\[
t(\bfx)=\|\bfA\bfx-\bfb\|_2^2+\lambda\|\bfx\|_2^2.
\]
Then
\[\frac{\hat q_k(\bfx_{k-1})-\hat q_k(\bfx_{k})}{\hat q_k(\bfx_{k})} \geq \frac{\max(\sigma_1^2(\bfU_{k+1}^{\dagger}),1)}{\min(\sigma_{k+1}^2(\bfU_{k+1}^{\dagger}),1)}-1 \Rightarrow t(\bfx_{k-1}) \geq t(\bfx_k) \]
where $\hat q_k(\bfx_{k-1})$ is the functional minimized at each iteration in the corresponding Krylov subspace, and where $\bfU_{k+1}$ is defined in  \eqref{eq:both_factorization}.
\end{proposition}
\begin{proof} At each iteration of H-FCMRH or H-FLSLU, the following functional in  \eqref{eq:hybrid_H_projection_approx} is minimized
\begin{equation*}
   \hat q_k(\bfx) = \| \bfU_{k+1}^{\dagger} (\bfA \bfx - \bfb) \|_2^2 +\lambda\|\bfx\|_2^2,
\end{equation*}
where $\bfU_{k+1}$ is defined in  \eqref{eq:both_factorization}. Then, we have that:
\begin{eqnarray*}
\underbrace{\min(\sigma_{k+1}^2(\bfU_{k+1}^{\dagger}),1)}_{k_2 \geq 0 } \,t(\bfx) \leq   \hat q_k(\bfx) \leq  \underbrace{\max(\sigma_1^2(\bfU_{k+1}^{\dagger}),1)}_{k_1 \geq 0 } \, t(\bfx).
\end{eqnarray*}
Using Lemma \ref{lemma1} with $h(\bfx)=t(\bfx)$ and $\hat h(\bfx) = \hat q_k(\bfx)$, the following are sufficient conditions for the solution at iteration $k$ to be a descent step in the Tikhonov-regularized function :
\[\frac{\hat q_k(\bfx_{k-1})-\hat q_k(\bfx_{k})}{\hat q_k(\bfx_{k})} \geq \frac{k_1-k_2}{k_2}=\frac{\max(\sigma_1^2(\bfU_{k+1}^{\dagger}),1)}{\min(\sigma_{k+1}^2(\bfU_{k+1}^{\dagger}),1)}-1.\]
\end{proof}
\begin{corollary} Given a sequence of solutions $\{\bfx_k\}_{k=1,..	}$ obtained with S\&S-H-FCMRH or S\&S-H-FLSLU, then,
\[
\frac{\hat q_k(\bfx_{k-1})-\hat q_k(\bfx_{k})}{\hat q_k(\bfx_{k})} \geq \kappa(\bfS_1)^2-1 \Rightarrow t(\bfx_{k-1})\geq  t(\bfx_k),
\]
where $\hat q_k(\bfx_{k-1})$ is the (sketched) functional \eqref{eq:H_sketched_projection_hybrid}, minimized at each iteration in the corresponding Krylov subspace. Moreover, since at iteration $k$, $\bfx_{k} \in \mathcal{R}(\bfZ_k)$ and $\bfx_{k-1} \in \mathcal{R}(\bfZ_{k-1})\subset \mathcal{R}(\bfZ_k)$, a tighter bound can be similarly derived by considering the singular values of $\bfS_1$ restricted to the range of $[\bfA \bfZ_k,\bfb]$. 
\end{corollary}
\begin{proof}
The proof can be readily adapted from the proof of Preposition \ref{prop2}. Similar bounds can also be found in \cite{Sabate2025rand}.
\end{proof}

\begin{proposition}\label{prop3} Given a sequence of solutions $\{\bfx_k\}_{k=1,..	}$ obtained with IRW-FCMRH or IRW-FLSLU, then,
for $f(\bfx)$ in \eqref{eq:functional} with a fixed regularization parameter $\lambda$ 
\[\frac{\hat q_k(\bfx_{k-1})-\hat q_k(\bfx_{k})}{\hat q_k(\bfx_{k})} \geq \frac{\max(\sigma_1^2(\bfU_{k+1}^{\dagger}),\sigma_1^2((\bfU_{k}^{LU})^{\dagger}))}{\min(\sigma_{k+1}^2(\bfU_{k+1}^{\dagger}),\sigma_k^2((\bfU_{k}^{LU})^{\dagger}))}-1 = \kappa(\bfC)^2-1 \Rightarrow f(\bfx_{k-1}) \geq f(\bfx_k) \]
where $\hat q_k(\bfx_{k-1})$ is the functional minimized at each iteration in the corresponding Krylov subspace, where 
$\bfU_{k+1}$ and $\bfU^{LU}_{k}$ are defined in  \eqref{eq:both_factorization} and \eqref{eq:IRW_H_projection_approx}, respectively, and $\bfC$ is a block diagonal matrix with $\bfU_{k+1}^{\dagger}$ and $(\bfU^{LU}_{k})^{\dagger}$ in its diagonal. 
\end{proposition}

\begin{proof}
Consider the quadratic tangent majorant  $q_k(\bfx)$ of $f(\bfx)$ at $\bfx_{k-1},$ 
\begin{equation*}
   q_k(\bfx) = \| \bfA \bfx - \bfb \|_2^2 + \lambda \|\bfW_k \bfL \bfx \|_2^2 + c_k, \quad \bfW_k =\bfW(\bfL \bfx_{k-1}),
\end{equation*}
where $c_k$ is a constant that does not depend on the solution and $\lambda$ has absorbed possible multiplicative factors. At each iteration of IRW-FCMRH or IRW-FLSLU, one minimizes \eqref{eq:IRW_H_projection_approx}, which corresponds to minimizing 
\begin{equation*}
   \hat q_k(\bfx) = \| \bfU_{k+1}^{\dagger} (\bfA \bfx - \bfb) \|_2^2 + \lambda \|(\bfU_{k}^{LU})^{\dagger} \bfW_k \bfL \bfx \|_2^2 + c_k, \quad \bfW_k =\bfW(\bfL \bfx_{k-1}),
\end{equation*}
for $\bfU_{k+1}$ defined in  \eqref{eq:both_factorization} and $\bfU^{LU}_{k}$ defined in \eqref{eq:IRW_H_projection_approx}. Then,
\begin{eqnarray*}
   \hat q_k(\bfx) \leq \sigma_1^2(\bfU_{k+1}^{\dagger})\, \| (\bfA \bfx - \bfb) \|_2^2 + \lambda \,\sigma_1^2((\bfU_{k}^{LU})^{\dagger} )\, \|\bfW_k \bfL \bfx \|_2^2  + c_k  \leq \underbrace{\max(\sigma_1^2(\bfU_{k+1}^{\dagger}),\sigma_1^2((\bfU_{k}^{LU})^{\dagger}))}_{k_1 \geq 0 } \, q_k(\bfx), \\
 \hat q_k(\bfx) \geq \sigma_{k+1}^2(\bfU_{k+1}^{\dagger})\, \| (\bfA \bfx - \bfb) \|_2^2 + \lambda \sigma_{k}^2((\bfU_{k}^{LU})^{\dagger}) \|\bfW_k \bfL \bfx \|_2^2  + c_k  \geq \underbrace{\min(\sigma_{k+1}^2(\bfU_{k+1}^{\dagger}),\sigma_k^2((\bfU_{k}^{LU})^{\dagger}))}_{k_2 \geq 0 } \, q_k(\bfx).
\end{eqnarray*}
Using Lemma \ref{lemma1} with $h(\bfx) = q_{k}(\bfx) $ and $\hat h(\bfx) = \hat q_{k}(\bfx) $, we have that:
\[\frac{\Delta_q}{\hat q_k(\bfx_{k})} \geq \frac{k_1-k_2}{k_2}= \frac{\max(\sigma_1^2(\bfU_{k+1}^{\dagger}),\sigma_1^2((\bfU_{k}^{LU})^{\dagger}))}{\min(\sigma_{k+1}^2(\bfU_{k+1}^{\dagger}),\sigma_k^2((\bfU_{k}^{LU})^{\dagger}))}-1 \Rightarrow  q_k(\bfx_{k-1}) - q_k(\bfx_{k}) \geq 0.\]
If this condition is met, then a descent step in the original function is also guaranteed:
\[f(\bfx_{k}) \leq q_k(\bfx_{k}) \leq q_k(\bfx_{k-1}) = f(\bfx_{k-1}), 
\]
where the first inequality holds since $ q_k(\bfx)$ is a majorant of $f(\bfx)$ and the last equality holds given the definition of $ q_k(\bfx)$. 
\end{proof}

\begin{corollary} Given a sequence of solutions $\{\bfx_k\}_{k=1,..	}$ obtained with S\&S-IRW-FCMRH or S\&S-IRW-FLSLU, then, for $f(\bfx)$ in \eqref{eq:functional} with a fixed regularization parameter $\lambda$ ,
\[
\frac{\hat q_k(\bfx_{k-1})-\hat q_k(\bfx_{k})}{\hat q_k(\bfx_{k})} \geq \kappa(\bfC)^2-1 \Rightarrow f(\bfx_k) \leq f(\bfx_{k-1}) 
\]
where $\hat q_k(\bfx_{k-1})$ is the (sketched) functional \eqref{eq:H_sketched_projection_IRW} minimized at each iteration in the corresponding Krylov subspace, and where and $\bfC$ is a block diagonal matrix with $\bfS_{1}$ and $\bfS_{2}$ in its diagonal. Moreover, since at iteration $k$, $\bfx_{k} \in \mathcal{R}(\bfZ_k)$ and $\bfx_{k-1} \in \mathcal{R}(\bfZ_{k-1})\subset \mathcal{R}(\bfZ_k)$, a tighter bound can be similarly derived by considering the singular values of $\bfC$ restricted to the range of $[[(\bfA \bfZ_k)^T;(\bfW_k \bfL \bfZ_k)^T]^T,[\bfb^T;\bf0^T]^T]$. 
\end{corollary}
\begin{proof}
The proof can be readily adapted from the proof of Preposition \ref{prop3}. Similar bounds can also be found in \cite{Sabate2025rand}.
\end{proof}
It is important to recall that flexible and sketch-and-solve methods (FCMRH, FLSLU, S2S-FCMRH, S2S-FLSLU), as well as their hybrid counterparts  ( H-FCMRH, H-FLSLU, S2S-H-FCMRH, S2S-H-FLSLU) are iterative regularization methods. This means that (some of) their regularization properties are associated to the qualities of the solution subspace, and therefore they only produce regularized solutions with the desired properties when equipped with early stopping, see e.g., \cite{Hansen2010}. 

On the other hand, iteratively reweighted (IRW-) methods have the potential to converge to the solution of the original regularized functional. In this case, the take home message is that, if a good estimation of the sufficient conditions was computable, or ideally one associated to even tighter bounds, one could use this information to trigger a restart in the inner-product free methods, or increase the size of the sketch in the sketch-and-solve methods. We leave this for future work.

\section{Numerical considerations}\label{sec55}
Up to this section, this work considers a general variational regularization term as presented in~\eqref{eq:functional}, where the specific definition of the weights depends on the variation regularization in \eqref{eq:var} that we are interested in. In this section, we recall the most commonly used regularization terms, the corresponding definition of the weights, and numerical considerations as they are implemented in the numerical experiments. As mentioned earlier, if the original regularization term is not differentiable at vectors with components equal to zero, we introduce a smoothing parameter directly in the definition of the weights. In this section, we also present a short summary of different regularization parameter choice criteria that we considered in the numerical experiments.

Note that further considerations, such as restarts, can be incorporated into this framework, see, e.g. \cite{Gazzola2021inexact,onisk2025restarts}, but 
are not included in this work.
\subsection{Sparsity promoting regularization}
Consider $\ell_p$ regularization, i.e. $R(\bfx)=\sfrac{1}{p}\|\bfL\|_p^p$ in \eqref{eq:var}, where $\bfL$ is an invertible change of variables. It is known that this regularization term, for $0<p\leq 1$, promotes sparse solutions \cite{Fornasier2010CS}. Note that, since $\mathcal{N}(\bfL)=\{\bf0\}$, we can guarantee the the solution is unique. In this case, we define the coefficients of the diagonal weight matrix in \eqref{eq:functional} as
\begin{equation}\label{eq:w_l1}
\bfW^{p, \tau}(\bfL \bfx)= \diag{(\bfz(1)^{2} + \tau^2)^{\frac{p-2}{4}},\dots, (\bfz(n)^{2}+ \tau^2)^{\frac{p-2}{4}}},
\end{equation}
for $\bfz = \bfL \bfx$, and where $\tau$ is the smoothing parameter that alleviates the lack of differentiability in vectors with $0$ coefficients. Note that the most commonly used choice is $p=1$, where $\ell_1$ regularization, also named LASSO in the statistics community, penalizes the absolute value of the coefficients. This is also the convex relaxation of the $\ell_p$ norm for any $p<1$. Recall that, at each iteration $k$ of a flexible methods, the weights in \eqref{eq:w_l1} are evaluated at the solution given at the previous iteration, i.e., $\bfW_k=\bfW^{p, \tau}(\bfL \bfx_{k-1})$. The iteration-dependent preconditioning, $(\bfW_k \bfL)^{-1}$, is well defined thanks to the smoothing parameter $\tau$, and it is easy to compute since $\bfW_k$ is a diagonal matrix and $\bfL$ is an invertible change of variables. 

\subsection{Total variation regularization}
Total Variation (TV) regularization is known to promote piece-wise constant solutions by encouraging sparsity in the discrete gradient of the solution \cite{rudin1992nonlinear}, and plays a fundamental role in correctly reconstructing edges in signal and image processing. The specific definition of the weights in \eqref{eq:functional} depends on the choice of the discrete approximation of a
differential operator, which depends on the dimension of the problem. 

For a 1D signal $\bfx\in\mathbb{R}^{n}$, we consider the following 1D discrete approximation of a
differential operator
\begin{equation}\label{1DD}
\bfD_1 = \begin{pmatrix}
1 & -1  &\\[-6pt]   %
&\hspace{-9pt}\fddots&\hspace{-9pt}\fddots \\[-6pt]   %
 & 1 & -1
\end{pmatrix} \in \mathbb{R}^{(n-1) \times n},
\end{equation}
so that the weights in \eqref{eq:functional} are defined as in \eqref{eq:w_l1} for $\bfL=\bfD_1$. In this case, $\bfW_k \bfL~\in~\mathbb{R}^{(n-1)\times n}$ is obviously not invertible, and its null space is that of constant vectors. Note that, alternatively, one can fix certain boundary conditions to avoid having to deal with a null-space. This is very common, for example, in CT applications, where 0 boundary conditions are a reasonable assumption, but we are not discussing this further in this work.  

Flexible methods for problems with non-invertible regularization matrices $\bfL$ require computing the A-weighted pseudoinverse as defined in \eqref{eq:Apseudo}, and this is needed either when solving the minimization problem in \eqref{eq:MM_w_std} or in \eqref{eq_F_square}. For $\bfL=\bfD_1$, the nullspace of $\bfW_k\bfL$ has only one vector, so the cost of computing the A-weighted pseudoinvese is dominated by that of computing the pseudoinverse, as defined in \eqref{eq:Apseudo}, and a matrix vector product with $\bfA$. Even if the latter, can sometimes be avoided, see \cite{Sabate2019TV}, the computation of $(\bfW_k\bfL)^{\dagger}$ can still be very expensive (this could be computed, for example, using an inner solve). Therefore, we sometimes consider the following approximation instead:
\begin{equation}\label{Apseudo_approx}
(\bfW_k\bfL)^{\dagger}_{\bfA} \approx \bfE \bfL^{\dagger} \bfW_k^{-1},
\end{equation}
where $\bfE$ is defined in \eqref{eq:Apseudo}. Using a flexible Krylov solver on \eqref{eq_F_square} also requires efficient matrix-vector products with $\bfP$, defined in \eqref{eq_F_square}. This can be done very efficiently by precomputing $\bfA\bfK$ ahead of the iterations. Note that flexible methods for \eqref{eq_F_square} promote sparsity through the preconditioning in the basis vectors. However, there isn't a clear interpretation of the regularization term in this case, so we always consider this problem without explicit regularization. 

Similar techniques can be used when considering total variation regularization for 2D or 3D, where more choice exists in the way vertical, horizontal and possibly depth derivatives are handled (i.e. isotropic or anisotropic TV), see e.g. \cite{Gazzola2021Edge, Gholami2024TV}. In this case, the computational gains from using the approximation \eqref{Apseudo_approx} are much stronger, given the structured nature of the discrete approximations of 2D and 3D differential operators.

Note that some inner-products might be performed in these solvers. First, computing $\bfx_A$ requires an inner-product before the computations start; then an inner-product is needed to do matrix-vector products with $\bfP$ and $\bfE$ in \eqref{eq_F_square} and  \eqref{eq:Apseudo}. However, both concern inner-products with a constant vector, and there is room for changing the normalization constant in case of overflow. Moreover, there is still no inner-products associated to the construction of the basis vectors.

\subsection{Regularization parameter choices}
One of the main advantages of working with projection methods is that a good regularization parameter can be chosen at each iteration for the small projected problem. This is at the core of hybrid regularization, see survey \cite{Chung2024survey}, where a (standard form) Tikhonov regularizer is added to the projected problem, but can also be extended to other methods that incorporate a (projected) regularization matrix, see e.g. \cite{buccini2022comparison, Gazzola2021IRW}.
Consider the general projected system, where the solution is considered as a function of the regularization parameter:
\begin{equation}\label{Asolve}
\bfy_k(\lambda) = \argmin \left\{ \|\bfA_k\bfy-\bfb_k\|_2^2+\lambda\|\bfL_k\bfy\|_2^2 \right\} = (\bfA_k^{T}\bfA_k+\lambda^2 \bfL_k^{T}\bfL_k)^{-1}\bfA_k^{T}\bfb_k =\bfA_{k,\lambda}^{\dagger} \bfb_k,
\end{equation} 
and where the dimensions of $\bfA_k$ and $\bfL_k$  are much smaller than those of $\bfA$ and $\bfL$. In this work, this might refer to the projected problem in any of the new flexible inner-product free methods, i.e., \eqref{eq:H_projection_approx},\eqref{eq:hybrid_H_projection_approx} or  \eqref{eq:IRW_H_projection_approx}; or their sketch-and-solve counterparts, i.e.,  \eqref{eq:H_sketched_projection},\eqref{eq:H_sketched_projection_hybrid} or  \eqref{eq:H_sketched_projection_IRW}.

In many applications it is reasonable to assume that an accurate estimation of the noise level is known, 
\begin{equation}\label{eq:noise_level}
\tau_\text{nl} = \|\bfe\|_2 / \|\bfA \bfx_{\rm true}\|_2 \approx  \|\bfe\|_2 / \|\bfb\|_2,    
\end{equation} 
where $\bfe$ is the noise in the right hand side. 

In this work we consider two different parameter choice rules. First, the discrepancy principle, where at each iteration, we compute 
\[
\lambda_k = \argmin_{\lambda} \left\{\|\bfA_k\bfy_k-\bfb_k\|_2^2 / \|\bfb\|_2^2 -\tau_\text{nl} \right\}.
\] 
Second, the weighted generalized cross validation (WGCV), where at each iteration, we compute 
\begin{equation}\label{eq:wgcv}
\lambda_k = \argmin_{\lambda} \frac{k\|\bfA_k \bfy_k(\lambda)- \bfb_k\|_2^2}{\trace{\bfI-\omega_k \bfA_k \bfA_{k,\lambda}^{\dagger} }^2} = \argmin_{\lambda} \frac{k\|(\bfA_k \bfA_{k,\lambda}^{\dagger}- \bfI) \bfb_k\|_2^2}{\trace{\bfI-\omega_k \bfA_k \bfA_{k,\lambda}^{\dagger} }^2},
\end{equation}
where $\bfA_{k,\lambda}^{\dagger}$ is defined in \eqref{Asolve}.
For different solvers one can use different weights $\omega_k$ in \eqref{eq:wgcv}. In particular, taking $\omega_k=1$ $\forall k$, corresponds to the GCV method. 

\section{Numerical experiments}\label{sec:numerics}
In this section, we show four experiments providing comparisons between the new methods and other standard solvers.

To compare the quality of the regularization parameter choices and study the potential of the solvers independently of these  criteria, we show the solutions for the optimal regularization parameter, computed at each iteration as:
\begin{equation} \label{para:opt}
\lambda_k=\argmin_{\lambda}\|\bfx_k(\lambda)-\bfx_{\text{true}}\|,
\end{equation}
where $\bfx_k(\lambda)$ is the solution at iteration $k$ for a given parameter $\lambda$.

For sketched-and-solve solvers, the sketching matrices $\bfS_1$ and $\bfS_2$ in \eqref{eq:H_sketched_projection_IRW}, are subsampling matrices, based on the approximation of leverage scores, as explained with more detail in \cite{Sabate2025rand}. Note that more work might be needed to find sketching techniques that are suitable for each application (e.g. that minimize communication).

The new flexible Krylov methods proposed in this work are summarized in Table \ref{tab:methods}, while other Krylov solvers are mentioned in Table \ref{tab:methods2} with their corresponding citations. Moreover, we also compare the new methods with the commonly used SpaRSA \cite{Wright2008sparsa} and FISTA \cite{beck2009fista}.

\begingroup
\renewcommand{\arraystretch}{1} 
\setlength{\tabcolsep}{4pt} 
\small
\begin{table}[h]
\caption{Summary of all the new flexible methods. Each column corresponds to the regularization term in the functional \eqref{eq:var} that decreases at each iteration, given conditions on the inexactness of the methods given in Section \ref{sec4}. In parenthesis, next to the solvers' names, the reference to the projected problem that is solved at each iteration. The left label corresponds to the method used to construct the basis for the solution subspace.}
\label{tab:methods}
\begin{tabular}{cccc}
& \bf $R(\bfx) =0$ & \bf $R(\bfx)=\|\bfx\|_2^2$  &  $R(\bfx)=\|W(\bfL\bfx)\bfL\bfx\|_2^2$   \\ \cline{2-4} 
\multicolumn{1}{l|}{\multirow{2}{*}{\bf{Flex. Hess.} \eqref{eq:Hessenberg_factorization}}} & FCMRH \eqref{eq:H_projection_approx}  & H-FCMRH \eqref{eq:hybrid_H_projection_approx}& IRW-FCMRH \eqref{eq:IRW_H_projection_approx} \\
\multicolumn{1}{l|}{ } &  S\&S-FCMRH \eqref{eq:H_sketched_projection} & S\&S-H-FCMRH \eqref{eq:H_sketched_projection_hybrid} &  S\&S-IRW-FCMRH \eqref{eq:H_sketched_projection_IRW} \\
\multicolumn{1}{l|}{\multirow{2}{*}{\bf{Flex. Gen. Hess.} \eqref{eq:Gen_Hessenberg_factorization}}} & FLSLU \eqref{eq:H_projection_approx} & H-FLSLU \eqref{eq:hybrid_H_projection_approx}& IRW-FLSLU  \eqref{eq:IRW_H_projection_approx} \\
\multicolumn{1}{l|}{ } &  S\&S-FLSLU  \eqref{eq:H_sketched_projection} & S\&S-H-FLSLU \eqref{eq:H_sketched_projection_hybrid} &  S\&S-IRW-FLSLU \eqref{eq:H_sketched_projection_IRW}
\end{tabular}
\end{table}
\normalsize
\endgroup

\begingroup
\renewcommand{\arraystretch}{1} 
\setlength{\tabcolsep}{4pt} 
\small
\begin{table}[h]
\caption{Summary of all the compared Krylov solvers. Each column corresponds to the regularization term in the functional \eqref{eq:var} that decreases at each iteration. Next to the solvers' names, one can find the corresponding references (either seminal or review papers). The left label corresponds to the method used to construct the basis for the solution subspace.}
\label{tab:methods2}
\begin{tabular}{cccc}
& \bf $R(\bfx) =0$ & \bf $R(\bfx)=\|\bfx\|_2^2$  &  $R(\bfx)=\|W(\bfL\bfx)\bfL\bfx\|_2^2$   \\ \cline{2-4} 
\multicolumn{1}{l|}{\bf Hessenberg } & CMRH \cite{sadok1999new} & H-CMRH \cite{brown2024hcmrh} &  \\
\multicolumn{1}{l|}{\bf{Arnoldi}} & GMRES \cite{Saad1886gmres} & H-GMRES \cite{Chung2024survey} &  \\
\multicolumn{1}{l|}{\bf{Flexible Arnoldi}} & FGMRES \cite{Gazzola2014GAT} & H-FGMRES \cite{Chung2024survey} & IRW-FGMRES \cite{Gazzola2021IRW} \\
\multicolumn{1}{l|}{\bf Golub Kahan Bidiag. (GKB)} & LSQR \cite{Saunders1882lsqr}  & H-LSQR \cite{Chung2024survey} & \\
\multicolumn{1}{l|}{\bf Flexible GKB} & FLSQR \cite{Chung2019lp} & H-FLSQR \cite{Chung2019lp} & IRW-FLSQR \cite{Gazzola2021IRW}
\end{tabular}
\end{table}
\normalsize
\endgroup

The computations were performed in MATLAB R2023a, and many solvers and test problems can be found in IR Tools \cite{IRtools} and AIR Tools II \cite{hansen2018air}. The low precision simulations are performed with the {\tt chop} function in MATLAB, which can be found in  \cite{higham2019simulating}.

\subsection{Deblurring stars problem} 
This deblurring example corresponds to the observation of stars under simulated atmospheric blur, and it is taken from \cite{RestoreTools,Nagy1998degraded}. Note that this is a non-homogeneous blur, but matrix-vector products with the system matrix can be efficiently performed. The true solution of the system, used to generate the problem, as well as the noisy measurements including additive Gaussian white noise with noise level $0.01$, can be found in Figure \ref{fig3}; where one can observe that the solution is qualitatively sparse. For this reason, we compare methods either including or based on $\ell_1$ regularization.

\begin{figure}[h]
\center
\includegraphics[width=\textwidth]{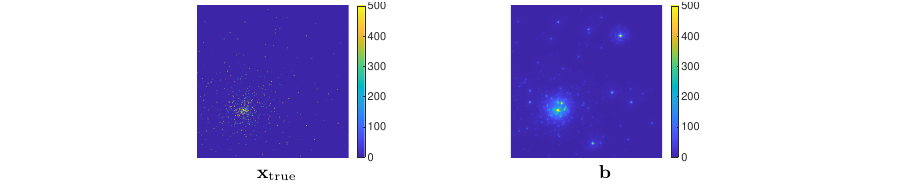}
\caption{Example 2. True solution modeling a starry night and noisy measurements simulating  atmospheric blur. \label{fig3}}
\end{figure}

To compare the performance of the new methods with respect to other standard solvers we present the relative error norms for different  regularizarion parameter choices.

In Figure \ref{fig4} (left), we show the results without explicit regularization. One can observe a small amount of semiconvergence, with the lowest relative error norms being achieved by the flexible methods. Figure \ref{fig4} (right) shows the results for a fixed regularization parameter that is good for the full dimensional problem, so that we can compare with standard solvers that require this parameter to be chosen ahead of the iterations. Note that all Krylov-based solvers achieve a similar relative error norm in significantly less iterations than FISTA and SpaRSA.

\begin{figure}[h]
\center
\includegraphics[width=\textwidth]{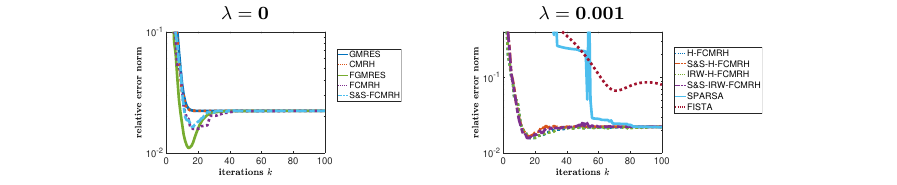}
\caption{Example 2. Relative error norms for different methods without explicit regularization (left) and for a fixed regularization parameter that is good for the full dimensional problem (right).\label{fig4}}
\end{figure}

In Figure \ref{fig5} we show the relative error norms if we allow the regularization parameter to change at each iteration. For each row of Figure \ref{fig5}, we compare the same methods  using different regularization parameter choices (note that a common legend is used on the whole row). Here, the new inner-product free solvers are displayed with a dotted line, the sketched-and-solve solvers with a dashed line, and the standard counterparts with a solid line. Recall that the GCV parameter choice corresponds to taking $\omega_k=1$ $\forall k$ in the definition of the WGCV \eqref{eq:wgcv}. Note that both FISTA and SpaRSA require a regularization parameter to be chosen before the iterations, so we use the one computed by H-FGMRES with the discrepancy principle at the end of the iterations.

Using the optimal regularization parameter choice as defined in \eqref{para:opt}, one can observe that the flexible methods are better than their standard counterparts, with the iteratively reweighted methods (IRW-) being better than the hybrid solvers (H-). Moreover, all Krylov solvers are significantly faster than other standard algorithms (FISTA and SpaRSA). Although sketched-and-solve solvers do not significantly improve the quality of the reconstructions with respect to the standard inner-product free methods when using the optimal regularization parameter, these become more relevant when choosing the regularization parameter using standard regularization parameter choices (discrepancy principle and GCV). This is because they approximate better the value of the residual norm at each iteration.

\begin{figure}[h]
\center
\includegraphics[width=\textwidth]{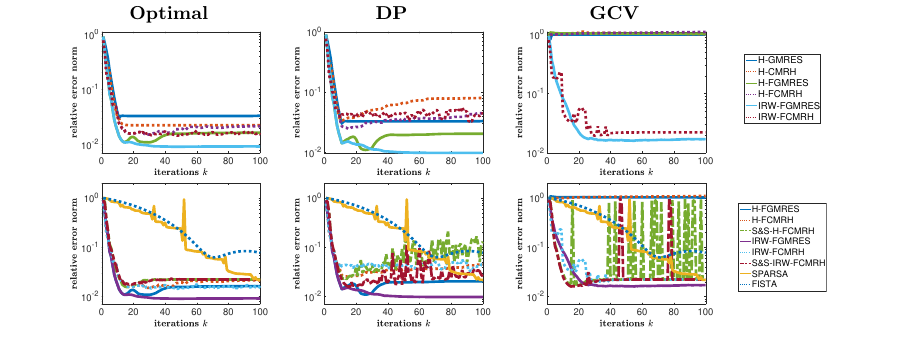}
\caption{Example 1. Relative error norms for different regularization parameter choices (each column), and for different methods. Note that both FISTA and SpaRSA require a regularization parameter to be chosen before the iterations, so we use the one computed by H-FGMRES with the discrepancy principle at the end of the iterations. Note that the labels are only displayed once for each row. \label{fig5}}
\end{figure}


\subsection{Computed tomography (CT) problem}
This example concerns a simulated computed tomography problem, where a sparse image is measured with a parallel beam 2D CT with with 362 rays and 18 uniformly distributed angles in the range $[1^\circ, 180^\circ]$, and Gaussian noise of noise level 0.1 is added to the measurements. The true solution, as well as the noisy measurements (a.k.a. sinogram) can be observed in Figure \ref{fig1}. Both the image and the measurements are generated using the IR Tools toolbox \cite{IRtools}. Since the solution in Figure \ref{fig1} is relatively sparse, we compare methods either including or based on $\ell_1$ regularization.

\begin{figure}[h]
\center
\includegraphics[width=\textwidth]{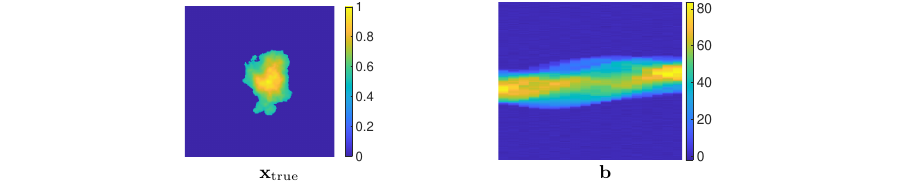}
\caption{Example 2. True solution and noisy measurements, also known as sinogram.\label{fig1}}
\end{figure}

In Figure \ref{fig2}, the relative error norm histories for different methods that compute the regularization parameter at each iteration are displayed. As specified in the title, each subfigure corresponds to a different regularization parameter choice, where the weights for the WGCV depend on the chosen method. In particular, $\omega_k=(k+1)/m$ as suggested in \cite{Renaut2017wgcv2} for LSLU-based methods, and $\omega_k$ is chosen as suggested in \cite{Chung2008wgcv} for LSQR-based methods –these have been chosen empirically to give the best results–.

Note that, without explicit regularization (first panel in Figure \ref{fig2}), the solvers based on the (possible flexible) generalized Hessenberg methods (LSLU, FLSLU) exhibit less semiconvergence than the standard counterparts (LSQR, FLSQR). This very favorable behaviour, however, could be attributed to the generalized Hessenberg methods adding basis vectors in the subspace that include less information with respect to  standard Krylov methods. Comparing the other regularization parameter choices, (panels 2-4 in Figure \ref{fig2}), one can observe that both the discrepancy principle and the weighted GCV do a good job at finding good regularization parameters (compared to the optimal parameter computed as in \eqref{para:opt}). It might be worth noting that the choice of the regularization matrix in the generalized Hessenberg methods does not seem to make any difference in terms of the relative error norm. Recall that, in the projected problem at each iteration, H-FLSLU considers standard Tikhonov while IRW-FLSLU considers the $U_{k}^{LU}$ matrix of an LU factorization of a reweighted term as in \eqref{eq:IRW_H_projection_approx}. This might indicate that this term is not a good approximation to the projected reweighted term.

Summarizing the results in Figure \ref{fig2}, all methods perform similarly, but the solvers based on the generalized Hessenberg method are inner-product free, and therefore have advantages with respect to the standard methods.
\begin{figure}[h]
\center
\includegraphics[width=\textwidth]{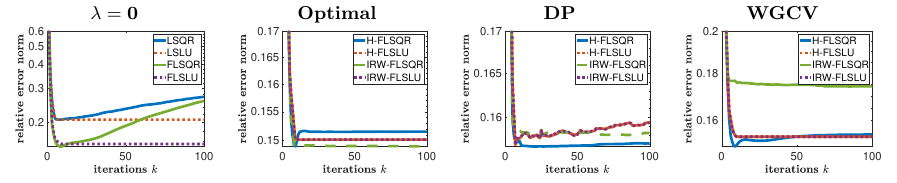}
\caption{Example 2. Relative error norms for different methods based on (flexible) GK and (flexible) generalized Hessenberg. \label{fig2}}
\end{figure}


\subsection{Example on low precision arithmetic}
In this example, a one-dimensional deblurring problem is solved in (simulated) low precision arithmetic. Note that the low precision simulations are done using the {\tt chop} function in MATLAB, see \cite{higham2019simulating}. The problem, also known as Spectra, features a sparse true signal simulating a simplified x-ray spectrum, and the system matrix corresponds to Gaussian blur. For more information, see \cite{Trussell1983Convergence}, and \cite{brown2024hcmrh} for a low precision version. In this example, we compare flexible methods where the weights correspond to those of $\ell_1$ regularization.

\begin{figure}[h]
\center
\includegraphics[width=\textwidth]{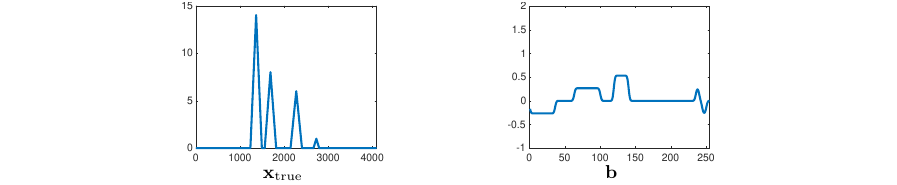}
\caption{Example 3. True solution and  measurements.\label{fig6}}
\end{figure}

Panel 1 and 3 of Figure \ref{fig7} show the relative error norms for different solvers for two different precisions: single precision, a.k.a. fp16, with 5 exponent and 10 significand bits, and q43, which has 4 exponent and 3 significand bits \cite{higham2019simulating}. Note that, for higher precision (panel 1 of Figure \ref{fig7}), traditional and inner-product free methods perform similarly, and in both cases the quality of the reconstructions improves when using flexible preconditioning. However, for lower precision (panel 3 of Figure \ref{fig7}), both traditional methods (GMRES and FGMRES) suffer from overflow in the norm computations in the first iteration, while CMRH and FCMRH produce solutions with decaying relative error norms. On panels 2 and 4 of Figure \ref{fig7}, one can observe the best  reconstructions in terms of error norm, produced by the solvers based on the Hessenberg method. One can see that, for both precisions, the flexible method produces a less oscillatory solution, with a better identification of the peaks. Since this experiment is done using simulated low precision, which requires modifying each floating point operation, only solvers without explicit regularization are compared.

\begin{figure}[h]
\center
\includegraphics[width=\textwidth]{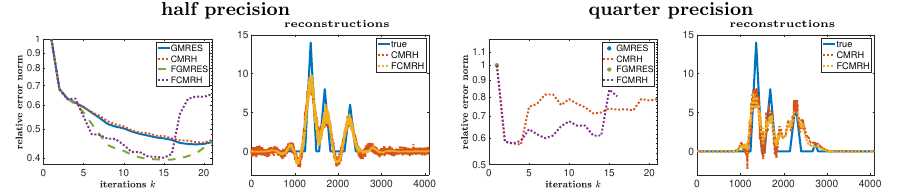}
\caption{Example 3. Experiments using simulated half and quarter precision arithmetic. Panel 1 and 3 display the relative error norms and panel 2 and 4 show the best reconstructions in terms of error norm, produced by the solvers based on the Hessenberg method. \label{fig7}}
\end{figure}


\subsection{Signal deblurring problem with total variation regularization}
This one dimensional example features an exact piece-wise signal $\bfx_{\text{true}}\in\mathbb{R}^{256}$ and a blurred noisy right-hand-side including white Gaussian noise with noise level 0.01, as can be observed in Figure~\ref{fig8}. Given the nature of this problem, we compare methods either including or based on TV regularization.

\begin{figure}[h]
\center
\includegraphics[width=\textwidth]{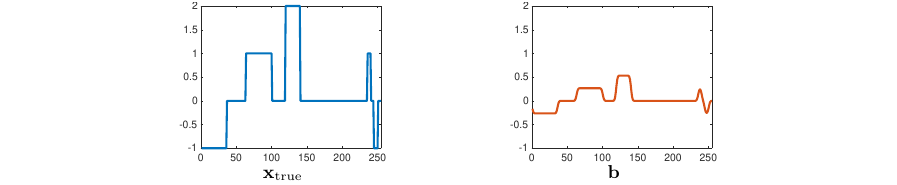}
\caption{Example 4. True solution and noisy measurements.\label{fig8}}
\end{figure}

The relative error norms for different methods can be observed in Figure~\ref{fig9}, both using the `exact pseudoinverse' of $\bfW_k \bfL$ in the computation of the A-weighted pseudoinverse in \eqref{eq_F_square} (computed using MATLAB's backslash), and the alternative approximation described in \eqref{Apseudo_approx}. Note that the results for both are similar. In this example, GMRES-D stands for a GMRES with a smoothing regularization term, namely a Tikhonov regularization term with a regularization matrix being the 1D discrete approximation of a differential operator in \eqref{1DD}. Comparing all methods, one can easily see that flexible methods outperform their non-flexible counterparts, with the inner-product free methods achieving almost the same reconstruction quality than their standard counterparts.

\begin{figure}[h]
\center
\includegraphics[width=\textwidth]{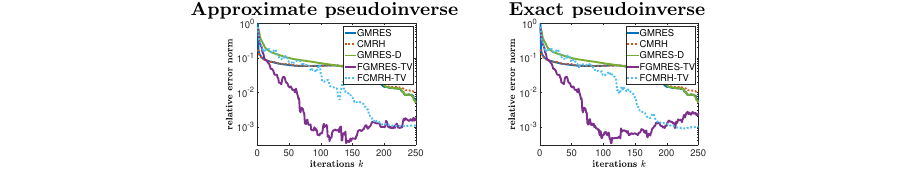}
\caption{Relative error norms for different methods. For the flexible methods, the iteration-dependent preconditioners require computing the pseudoinverse of $\bfW_k \bfL$. This is done `exactly' using MATLAB's backslash (left), and using the alternative approximation in \eqref{Apseudo_approx} (right). \label{fig9}}
\end{figure}

In Figure~\ref{fig10}, we show the best reconstructions for the different compared methods. Here, the flexible methods have been chosen to use the approximation in \eqref{Apseudo_approx}. To illustrate part of the success of flexible methods in this case,  in Figure~\ref{fig11} we further display a selection of basis vectors for the solution subspace used by the same solvers. One can easily see that flexible methods use a subspace for the solution that is constructed with basis vectors that are qualitatively  much more piece-wise constant. This helps mitigating the ringing effect that might appear in the reconstructions near the edges of the signal. 

\begin{figure}[h]
\center
\includegraphics[width=\textwidth]{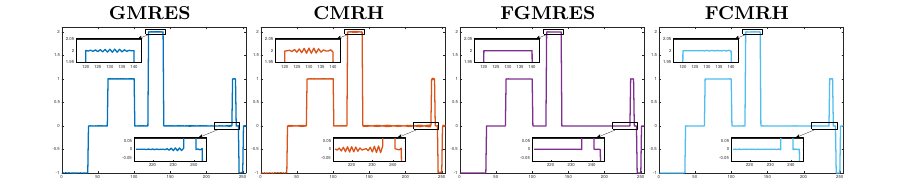}
\caption{Example 4. (Best) reconstructions obtained using different solvers.\label{fig10}}
\end{figure}

\begin{figure}[h]
\center
\includegraphics[width=\textwidth]{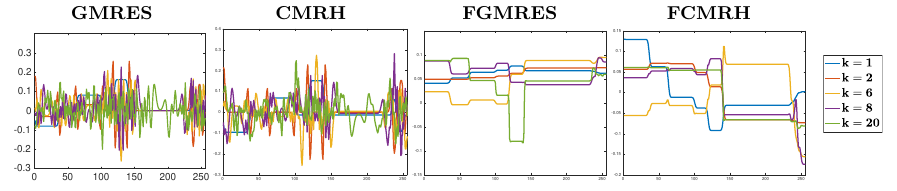}
\caption{Example 4. Selection of basis vectors for the solution subspace used by the different projection methods.\label{fig11}}
\end{figure}

\section{Conclusions and future work}\label{sec:conclusions}

This paper introduces new flexible and inner-product free Krylov methods with iteration-dependent preconditioning, as well sketch-and-project variants. This is presented in the framework of linear discrete inverse problems, where we approximate solutions of  least-squares problem with a general variational term, for example $\ell_1$ or total variation. The proposed solvers are flexible projection methods, and the key of their success is that the iteration-dependent preconditioning encodes prior information about the solution, that is related to the choice of the regularizer. Therefore, the solution subspaces are tailored to the solution. 
 
The novelty in this work corresponds to exploiting the idea behind flexible   Krylov solvers, without using inner-products in the core of the algorithm. Therefore, we use flexible CMRH (FCMRH) in this context for the first time and we introduce a completely new flexible LSLU (FLSLU).  Moreover, we propose new variants by modifying the regularization term in the projected problems, namely hybrid (H-FCMRH and H-FLSLU) and iteratively-reweighted methods (IRW-FCMRH and IRW-FLSLU) which determines the convergence of the methods without early stopping.  Last, we present sketch-and-solve variants, namely (S\&S-FCMRH, S\&S-FLSLU, S\&S-H-FCMRH, S\&S-H-FLSLU, S\&S-IRW-FCMRH and S\&S-IRW-FLSLU).
 
This work presents several numerical examples, where the flexible inner-product free variants show to be competitive with respect to their standard flexible counterparts. 
Moreover, some problems highlight  different aspects of the new solvers.
For example, we show a low precision setting, where it can be observed that flexible inner-product free Krylov solvers are less prone to problems associated to computations in low precision, often appearing when computing inner-products, while delivering comparable results to their standard counterparts. We also show that the effect of the variable preconditioning in the basis vectors for the solution is similar for the flexible inner-product free and the flexible standard methods.

On the other hand, sketch-and-solve methods can be particularly useful to  find good regularization parameters using standard regularization parameter choice criteria, since they do a better job at approximating the residual norm as well as the value of the regularization term. New  regularization parameter choice criteria that are tailored for inner-product free methods are still missing.

Last, none of the presented methods has convergence guarantees given the inexact nature of their optimality conditions. However, we can say something about the monotonicity properties of the minimized functionals under some conditions, and future work could include restarting (for flexible inner-product free methods) or increasing the size of the sketch (for sketch-and-solve methods), to address this. Last, an alternative way of using sketching in this context, which would provide convergence guarantees for the price of including inner-products (in practice, for a smaller dimension), is the sketch-and-precondition framework. This could be used to solve the tall and skinny problems appearing when restricting the solution to belong the the subspaces generate using either the flexible Hessenberg of the flexible generalized Hessenberg methods, but this is left for future work.

\bibliographystyle{abbrv}
\bibliography{biblio}

\begin{thebibliography}{10}

\bibitem{beck2009fista}
A.~Beck and M.~Teboulle.
\newblock A fast iterative shrinkage-thresholding algorithm for linear inverse
  problems.
\newblock {\em SIAM J. Imaging Sci.}, 2(1):183--202, 2009.

\bibitem{brown2025hlslu}
A.~N. Brown, J.~Chung, J.~G. Nagy, and M.~Sabat\'{e}~Landman.
\newblock Inner-product free {K}rylov methods for large-scale inverse problems.
\newblock {\em SIAM J. Sci. Comput.}, 0(0):S161--S182, 0.

\bibitem{brown2024hcmrh}
A.~N. Brown, M.~{Sabat\'e Landman}, and J.~G. Nagy.
\newblock {H-CMRH}: An inner product free hybrid {K}rylov method for
  large-scale inverse problems.
\newblock {\em SIAM J. Matrix Anal. Appl.}, 46(1):232--255, 2025.

\bibitem{buccini2022comparison}
A.~Buccini, M.~Pragliola, L.~Reichel, and F.~Sgallari.
\newblock A comparison of parameter choice rules for $\ell_p$- $\ell_q$
  minimization.
\newblock {\em Ann Univ Ferrara}, 68(2):441--463, 2022.

\bibitem{Buccini2023restart}
A.~Buccini and L.~Reichel.
\newblock Limited memory restarted $\ell_p$-$\ell_q$ minimization methods using
  generalized {K}rylov subspaces.
\newblock {\em Adv. Comput. Math.}, 49, 04 2023.

\bibitem{calvetti2005priorconditioners}
D.~Calvetti and E.~Somersalo.
\newblock Prior conditioners for linear systems.
\newblock {\em Inverse Probl.}, 21(4):1397--1418, 2005.

\bibitem{Chung2019lp}
J.~Chung and S.~Gazzola.
\newblock Flexible {K}rylov methods for $\ell_p$ regularization.
\newblock {\em SIAM J. Sci. Comput.}, 41(5):S149--S171, 2019.

\bibitem{Chung2024survey}
J.~Chung and S.~Gazzola.
\newblock Computational methods for large-scale inverse problems: A survey on
  hybrid projection methods.
\newblock {\em SIAM Rev.}, 66(2):205--284, 2024.

\bibitem{chung2025randomized}
J.~Chung and S.~Gazzola.
\newblock Randomized {K}rylov methods for inverse problems.
\newblock {\em arXiv:2508.20269}, 2025.

\bibitem{Chung2008wgcv}
J.~Chung, J.~G. Nagy, and D.~O'Leary.
\newblock A weighted-{GCV} method fo {L}anczos-hybrid regularization.
\newblock {\em ETNA}, 28, 01 2008.

\bibitem{Fornasier2010CS}
M.~Fornasier and H.~Rauhut.
\newblock Compressive sensing.
\newblock {\em Handbook of Mathematical Methods in Imaging}, 05 2010.

\bibitem{IRtools}
S.~Gazzola, P.~C. Hansen, and J.~G. Nagy.
\newblock {IR Tools}: a {MATLAB} package of iterative regularization methods
  and large-scale test problems.
\newblock {\em Numer. Algorithms}, 81(3):773--811, 2019.
\newblock Software available at \url{https://github.com/jnagy1/IRtools}.

\bibitem{Gazzola2014GAT}
S.~Gazzola and J.~G. Nagy.
\newblock Generalized {A}rnoldi--{T}ikhonov method for sparse reconstruction.
\newblock {\em SIAM J. Sci. Comput.}, 36(2):B225--B247, 2014.

\bibitem{Gazzola2021IRW}
S.~Gazzola, J.~G. Nagy, and M.~S. Landman.
\newblock Iteratively reweighted {FGMRES} and {FLSQR} for sparse
  reconstruction.
\newblock {\em SIAM J. Sci. Comput.}, 43(5):S47--S69, 2021.

\bibitem{Gazzola2021inexact}
S.~Gazzola and M.~Sabat\'{e}~Landman.
\newblock Regularization by inexact {K}rylov methods with applications to blind
  deblurring.
\newblock {\em SIAM J. Matrix Anal. Appl.}, 42(4):1528--1552, 2021.

\bibitem{Sabate2019TV}
S.~Gazzola and M.~{Sabaté Landman}.
\newblock Flexible {GMRES} for total variation regularization.
\newblock {\em BIT}, 59(3):721--746, Sept. 2019.

\bibitem{Gazzola2021Edge}
S.~Gazzola, S.~J. Scott, and A.~Spence.
\newblock Flexible {K}rylov methods for edge enhancement in imaging.
\newblock {\em J. Imaging}, 7(10), 2021.

\bibitem{Gholami2024TV}
A.~Gholami and S.~Gazzola.
\newblock Robust estimation of structural orientation parameters and 2{D}/3{D}
  local anisotropic {T}ikhonov regularization.
\newblock {\em Geophysics}, 89(6):V521--V536, 2024.

\bibitem{Gu2020FCMRH}
X.-M. Gu, T.-Z. Huang, B.~Carpentieri, A.~Imakura, K.~Zhang, and L.~Du.
\newblock Efficient variants of the {CMRH} method for solving a sequence of
  multi-shifted non-{H}ermitian linear systems simultaneously.
\newblock {\em J Appl Math Comput}, 375:112788, Sept. 2020.

\bibitem{Hansen2010}
P.~C. Hansen.
\newblock {\em Discrete Inverse Problems: Insight and Algorithms}.
\newblock SIAM, Philadelphia, 2010.

\bibitem{hansen2007smoothing}
P.~C. Hansen and T.~K. Jensen.
\newblock Smoothing-norm preconditioning for regularizing minimum-residual
  methods.
\newblock {\em SIAM J. Matrix Anal. Appl.}, 29(1):1--14, 2007.

\bibitem{hansen2018air}
P.~C. Hansen and J.~S. J{\o}rgensen.
\newblock {AIR Tools II}: {A}lgebraic {I}terative {R}econstruction {M}ethods,
  {I}mproved {I}mplementation.
\newblock {\em Numer. Algorithms}, 2018.
\newblock Sotfware available at \url{https://github.com/jakobsj/AIRToolsII/}.

\bibitem{hastings2025PICCS}
F.~V. Hastings, S.~M. R.~S. Islam, M.~{Sabaté Landman}, S.~Hatamikia, C.-B.
  Schönlieb, and A.~Biguri.
\newblock Real-time {CBCT} reconstructions using {K}rylov solvers in repeated
  scanning procedures.
\newblock {\em arXiv:2509.08574}, 2025.

\bibitem{hatamikia2023source}
S.~Hatamikia, A.~Biguri, G.~Kronreif, T.~Russ, J.~Kettenbach, and
  W.~Birkfellner.
\newblock Source-detector trajectory optimization for {{CBCT}} metal artifact
  reduction based on {PICCS} reconstruction.
\newblock {\em Zeitschrift f\"ur Medizinische Physik}, 34(4), 2023.

\bibitem{Hessenberg1940}
K.~Hessenberg.
\newblock Behandlung linearer {E}igenwertaufgaben mit {H}ilfe der
  {H}amilton-{C}ayleyschen {G}leichung.
\newblock Tech. Rep. 1) Bericht der Reihe Numerische Verfahren, Institut für
  Praktische Mathematik, Technische Hochschule Darmstadt, 1940.

\bibitem{higham2019simulating}
N.~J. Higham and S.~Pranesh.
\newblock Simulating low precision floating-point arithmetic.
\newblock {\em SIAM J. Sci. Comput.}, 41(5):C585--C602, 2019.

\bibitem{Huang2017MM}
G.~Huang, A.~Lanza, S.~Morigi, L.~Reichel, and F.~Sgallari.
\newblock Majorization–minimization generalized {K}rylov subspace methods for
  $\ell_p$ – $\ell_q$ optimization applied to image restoration.
\newblock {\em BIT}, 57, 01 2017.

\bibitem{Lanza2015GKS}
A.~Lanza, S.~Morigi, L.~Reichel, and F.~Sgallari.
\newblock A generalized {K}rylov subspace method for $\ell_p$-$\ell_q$
  minimization.
\newblock {\em SIAM J. Sci. Comput.}, 37(5):S30--S50, 2015.

\bibitem{Kolda2022sampling}
B.~W. Larsen and T.~G. Kolda.
\newblock Sketching matrix least squares via leverage scores estimates.
\newblock {\em arXiv:2201.10638}, 2022.

\bibitem{Martinsson_Tropp_2020}
P.-G. Martinsson and J.~A. Tropp.
\newblock Randomized numerical linear algebra: Foundations and algorithms.
\newblock {\em Acta Numer.}, 29:403–572, 2020.

\bibitem{Nagy1998degraded}
J.~G. Nagy and D.~P. O'{L}eary.
\newblock Restoring images degraded by spatially-variant blur.
\newblock {\em SIAM J. Sci. Comput.}, 19:1063--1082, 1998.

\bibitem{RestoreTools}
J.~G. Nagy, K.~Palmer, and L.~Perrone.
\newblock Iterative methods for image deblurring: {A} {MATLAB} object-oriented
  approach.
\newblock {\em Numer. Algorithms}, 36:73–93, 2004.

\bibitem{onisk2025restarts}
L.~Onisk and M.~{Sabaté Landman}.
\newblock Iterative refinement and flexible iteratively reweighed solvers for
  linear inverse problems with sparse solutions.
\newblock {\em arXiv:2502.02303}, 2025.

\bibitem{Saunders1882lsqr}
C.~C. Paige and M.~A. Saunders.
\newblock Lsqr: An algorithm for sparse linear equations and sparse least
  squares.
\newblock {\em ACM Trans. Math. Softw.}, 8(1):43–71, Mar. 1982.

\bibitem{Renaut2017wgcv2}
R.~A. Renaut, S.~Vatankhah, and V.~E. Ardestani.
\newblock Hybrid and iteratively reweighted regularization by unbiased
  predictive risk and weighted {G}{C}{V} for projected systems.
\newblock {\em SIAM J. Sci. Comput.}, 39(2):B221--B243, 2017.

\bibitem{rudin1992nonlinear}
L.~I. Rudin, S.~Osher, and E.~Fatemi.
\newblock Nonlinear total variation based noise removal algorithms.
\newblock {\em Phys. D}, 60(1-4):259--268, 1992.

\bibitem{Saad1886gmres}
Y.~Saad and M.~H. Schultz.
\newblock {GMRES}: a generalized minimal residual algorithm for solving
  nonsymmetric linear systems.
\newblock {\em SIAM J. Sci. Stat. Comput.}, 7(3):856–869, July 1986.

\bibitem{Sabate2025randCMRH}
M.~Sabat{\'e}~Landman, A.~N. Brown, J.~Chung, and J.~G. Nagy.
\newblock Randomized and inner-product free {K}rylov methods for large-scale
  inverse problems.
\newblock {\em Numer. Algorithms}, In Press, 2025.

\bibitem{Sabate2025rand}
M.~{Sabat\'e Landman} and Y.~Nakatsukasa.
\newblock Randomized flexible {K}rylov methods for $\ell_p$ regularization.
\newblock {\em arXiv:2510.11237}, 2025.

\bibitem{sadok1999new}
H.~Sadok.
\newblock {CMRH}: {A} new method for solving nonsymmetric linear systems based
  on the {H}essenberg reduction algorithm.
\newblock {\em Numer. Algor.}, 20:303--321, 1999.

\bibitem{sadok2012new}
H.~Sadok and D.~B. Szyld.
\newblock {A} new look at {CMRH} and its relation to {GMRES}.
\newblock {\em {BIT}}, 52:485--501, 2012.

\bibitem{Stayman2013PIRPLE}
J.~W. Stayman, H.~Dang, Y.~Ding, and J.~H. Siewerdsen.
\newblock {PIRPLE}: a penalized-likelihood framework for incorporation of prior
  images in {CT} reconstruction.
\newblock {\em Phys. Med. Biol.}, 58(21):7563, oct 2013.

\bibitem{Trussell1983Convergence}
H.~J. Trussell.
\newblock Convergence criteria for iterative restoration methods.
\newblock {\em {IEEE} Transactions on Acoustics, Speech, and Signal
  Processing}, ASSP-31(1):129--136, 02 1983.

\bibitem{Wright2008sparsa}
S.~J. Wright, R.~D. Nowak, and M.~A.~T. Figueiredo.
\newblock Sparse reconstruction by separable approximation.
\newblock In {\em 2008 IEEE International Conference on Acoustics, Speech and
  Signal Processing}, pages 3373--3376, 2008.

\bibitem{Zhang2013CMRHflexible}
K.~Zhang and C.~Gu.
\newblock A flexible {CMRH} algorithm for nonsymmetric linear systems.
\newblock {\em J Appl Math Comput}, 45, 2013.

\end{thebibliography}

\end{document}